\documentclass[11pt,leqno,twoside]{amsart}
\usepackage{amssymb,amsmath,amsthm,color}
\usepackage{t1enc}
\usepackage{a4,indentfirst,latexsym}
\usepackage{graphics}
\usepackage{cite,enumitem,graphicx}
\usepackage[colorlinks=true,urlcolor=blue,
citecolor=red,linkcolor=blue,linktocpage,pdfpagelabels,
bookmarksnumbered,bookmarksopen,hypertexnames=false]{hyperref}
\usepackage[english]{babel}
\usepackage[left=2.61cm,right=2.61cm,top=2.72cm,bottom=2.72cm]{geometry}
\usepackage[hyperpageref]{backref}
\usepackage[colorinlistoftodos]{todonotes}

\makeatletter
\providecommand\@dotsep{5}
\def\listtodoname{List of Todos}
\def\listoftodos{\@starttoc{tdo}\listtodoname}
\makeatother

\numberwithin{equation}{section}
\newtheorem{theorem}{Theorem}[section]
\newtheorem{proposition}[theorem]{Proposition}

\newtheorem{lemma}[theorem]{Lemma}
\newtheorem{example}[theorem]{Example}

\newtheorem{remark}{Remark}

\def\div{\mathop{\mathrm{div}\,}}

\def\N{\mathbb{N}}
\def\R{\mathbb{R}}

\def\R {{\rm I}\hskip -0.85mm{\rm R}}
\def\N {{\rm I}\hskip -0.85mm{\rm N}}

\title[A new Minimax Theorem]
{A New Minimax Theorem for a Class of $ C^1$-Functionals}

\author[D. S. Pereira]{Denilson S. Pereira}
\author[J. R. Santos Jr.]{Jo\~ao R. Santos J\'unior}
\author[J. F. F. Silva]{Jo\~ao F. F. Silva}

\address[D. S. Pereira]
{\newline\indent Unidade Acad\^{e}mica de Matem\'atica,
\newline\indent
Universidade Federal de Campina Grande
\newline\indent
58429-900, Campina Grande, PB, Brazil}
\email{\href{mailto: denilsonsp@mat.ufcg.edu.br}{denilsonsp@mat.ufcg.edu.br}}

\address[J. R. Santos Jr.]
{\newline\indent Faculdade de Matem\'atica
\newline\indent
Instituto de Ci\^{e}ncias Exatas e Naturais
\newline\indent
Universidade Federal do Par\'a
\newline\indent
Avenida Augusto corr\^{e}a 01, 66075-110, Bel\'em, PA, Brazil}
\email{\href{mailto: joaojunior@ufpa.br }{joaojunior@ufpa.br}}

\address[J. F. F. Silva]
{\newline\indent Programa de Doutorado em Matem\'atica
\newline\indent
Instituto de Ci\^{e}ncias Exatas e Naturais
\newline\indent
Universidade Federal do Par\'a
\newline\indent
Avenida Augusto corr\^{e}a 01, 66075-110, Bel\'em, PA, Brazil}
\email{\href{mailto: joaofelipe.ce@gmail.com}{joaofelipe.ce@gmail.com}}

\thanks{Jo\~ao R. Santos was partially
supported by CNPq/Brazil 304340/2025-1, Brazil. }
\thanks{{\it Data availability statement}: Data sharing not applicable to this article as no datasets were generated or analysed during the current study.
{\it Conflict of interest statement}: The authors declare that they have no conflict of interest.}

\subjclass[2010]{ 
35J50,  	
35J57. 
	}
\keywords{Ground-state solution, Nehari method, Open cones}

\begin{document}

\begin{abstract}
We establish a new abstract minimax theorem for obtaining ground-state critical points of a class of \(C^{1}\)-functionals. The central feature of this class is that the associated Nehari set is not necessarily a differentiable manifold, but is homeomorphic to a (possibly non-complete) submanifold of the unit sphere. This occurs for functionals whose fibering maps \(t \mapsto I(tu)\) possess a unique global maximum point only when \(u\) belongs to an open cone in the underlying Banach space. A lower semicontinuous boundary barrier, which may be finite, controls minimizing sequences near the boundary of the spherical cone. As a consequence, we provide a general minimax characterization of the ground-state energy level. Elliptic applications complement existing results and include superlinear and quasilinear problems with finite barriers, including thresholds depending on the profile of a first eigenfunction and its gradient.
\end{abstract}


\bigskip

\maketitle
\begin{center}
\begin{minipage}{12cm}
\tableofcontents
\end{minipage}
\end{center}

\bigskip

\section{Introduction}

In \cite{SW}, among other things,  A. Szulkin and T. Weth study elliptic partial differential equations with continuous (and possibly nondifferentiable) nonlinearities whose primitives satisfy the so-called superquadratic condition at infinity. In order to apply the Nehari method to find ground-state and other kinds of solutions, the authors introduce an interesting approach to overcome the lack of a $C^1$-structure for the Nehari set. In fact, under suitable conditions on the functional $I:X\to\R$ and on the Banach space $X$, the existence of a homeomorphism is proved between the Nehari set $\mathcal{N}$ associated with $I$ and the unit sphere $\mathcal{S}$ of $X$. This homeomorphism allows them to define an auxiliary functional $\Psi:\mathcal{S}\to\R$ with the convenient property that the existence of critical points for $\Psi$ implies the existence of critical points of $I$. Due to conditions imposed on the Banach space $X$, the unit sphere $\mathcal{S}$ is a $C^1$ manifold, and for this reason, the task of looking for critical points of $\Psi$ is a more treatable problem.

\medskip

Still in \cite{SW}, as an application of the referred approach, a classical paper by Benci and Cerami (see \cite{BC}) is revisited, which relates the number of solutions of a certain elliptic PDE with the topology of the domain where the equation is considered. The authors are able to improve the main results in \cite{BC} by assuming weaker conditions on the nonlinearity. However, since the idea is to obtain positive solutions, the Nehari set in this case is no longer homeomorphic to the whole unit sphere $\mathcal{S}$, but to an open subset $\mathcal{S}^+$ of $\mathcal{S}$. As observed in \cite{SW}, this new situation brings together some technical difficulties, as, for instance, the fact that $\mathcal{S}^+$ has a nonempty boundary in $\mathcal{S}$ and, therefore, the behaviour of minimizing sequences needs to be carefully controlled near the boundary.

\medskip

In this article, we are interested in an abstract setting in which the situation considered in \cite{SW}, namely, the possible existence of a homeomorphism between $\mathcal{N}$ and a proper subset of $\mathcal{S}$, is a particular case. Indeed, it is clear that if not all fibering maps $t\mapsto I(tu)$ have an appropriate behavior, which, in the context of \cite{SW}, means they have a unique global maximum point with positive energy (see hypothesis ($A_2$) in \cite{SW}), we should not expect the Nehari set to be homeomorphic to the unit sphere of the space. In such cases the full-sphere formulation in \cite{SW} does not apply directly. However, \cite[Section~3.4]{SW} already treats a proper open subset of the sphere when the reduced energy diverges at its boundary. The distinction considered here is an abstract lower semicontinuous boundary barrier $\Gamma$, which may be finite, together with the strict energy separation $c_{\mathcal N}<\inf_{\partial\mathcal S_{\mathcal A}}\Gamma$. This permits the variational construction on the closure of the spherical cone in the finite-barrier case as well. 

The extension is substantive in the finite-barrier case. Subsection~\ref{asymptotic-p}, including Example~\ref{p-explicit-primitive}, gives an asymptotically $p$-linear reaction satisfying our strict energy separation with a finite $\Gamma$. Subsection~\ref{finite-cubic} gives a different mechanism, involving a sign-changing cubic reaction, for which Proposition~\ref{cubic-boundary} constructs
\[
w_n\in\mathcal S_{\mathcal A},\qquad
w_n\to w\in\partial\mathcal S_{\mathcal A},\qquad
\Psi(w_n)\to\Gamma(w)<+\infty.
\]
Subsection~\ref{gradient-barrier} further gives a quasilinear example whose finite threshold depends on the spatial profile of $e_1$ and $\nabla e_1$, rather than only on a constant saturation value and a support measure. The boundary-divergence property of \cite[Section~3.4, Lemma~26]{SW} is unavailable in both these applications. Our hypothesis $(H)$ instead separates the minimizing level from the boundary by $c_{\mathcal N}<\inf\Gamma$. The distinction is the admissibility of finite boundary energies, rather than the existence of a nonempty boundary alone, and the cubic example shows that it is not restricted to asymptotically linear reactions.

\medskip

It turns out that in numerous relevant variational problems, the associated energy functional does not exhibit homogeneous behavior in all directions. In fact, in what follows, we identify some simple prototypes of PDEs in which the existence of a unique global maximum point with positive energy for the mapping $t \mapsto I(tu)$ occurs only for a certain class of functions $u \in X$ (which we denote by $\mathcal{A}$):

\medskip 

\begin{enumerate}
\item Problems involving sign-changing weights: 

\begin{equation}
\left \{ \begin{array}{ll}
   -\Delta u=a(x)|u|^{p-2}u & \mbox{in $\Omega$,}\\
  u=0 & \mbox{on $\partial\Omega$,}
\end{array}\right.
\end{equation}
where $a\in L^{\infty}(\Omega)$, $a^\pm\neq 0$, $2<p<2^{\ast}$ and $\Omega$ is a bounded domain in $\R^N$. The associated energy functional is
$$
I(u)=\frac{1}{2}\int_{\Omega}|\nabla u|^{2}dx-\frac{1}{p}\int_{\Omega}a(x)|u|^{p}dx.
$$
If $u$ belongs to
$$
\mathcal{A}=\left\{ u\in H_{0}^{1}(\Omega): \int_{\Omega}a(x)|u|^{p} dx>0\right\},
$$
then the fibering map $[0, +\infty)\ni t\mapsto I(tu)$ has a unique critical point, which is a global maximum. However, it has no critical points if $u\not\in \mathcal{A}$.

\medskip

\item Asymptotically linear elliptic problems:

\begin{equation}
\left \{ \begin{array}{ll}
   -\Delta u=f(u) & \mbox{in $\Omega$,}\\
  u=0 & \mbox{on $\partial\Omega$,}
\end{array}\right.
\end{equation}
where $\Omega$ is bounded domain in $\R^N$ and $f \in C(\R)$ verifies:\\
\begin{itemize}
\item[$(f1)$]  
$$
0<\lim_{s \rightarrow 0}2F(s)/s^{2}=\alpha<\lambda_1<\lim_{\vert s\vert \rightarrow +\infty}2F(s)/s^{2}=\eta,
$$ 
where $\lambda_1$ denotes the first eigenvalue of the Laplacian operator with Dirichlet boundary condition and $F(t)=\int_{0}^{t}f(s)ds$;
\item[$(f2)$] The map $s \mapsto f(s)/s$ is strictly increasing on $(0,+\infty)$ and strictly decreasing on $(-\infty,0)$. 
\end{itemize}
\medskip
The associated energy functional is
$$
I(u)=\frac{1}{2}\int_{\Omega}|\nabla u|^{2}dx-\int_{\Omega}F(u)dx.
$$
Note that, for each $u\neq 0$,
\begin{equation}
\frac{I(tu)}{t^{2}}= \frac{1}{2}\|u\|^{2}-\int_{[u\neq 0]}\left[\frac{F(tu)}{(tu)^{2}}\right]u^{2}dx.
\end{equation}
Consequently, from $(f1)$,
$$
\lim_{t\to 0^+}\frac{I(tu)}{t^{2}}=\frac{1}{2}\left(\|u\|^{2}-\alpha\int_{\Omega}u^{2}dx\right)\geq \frac{1}{2}\left(1-\frac{\alpha}{\lambda_1}\right)\|u\|^{2}>0
$$
and
$$
\lim_{t\to \infty}\frac{I(tu)}{t^{2}}=\frac{1}{2}\left(\|u\|^{2}-\eta\int_{\Omega}u^{2}dx\right)<0,
$$
if $u$ belongs to the set
$$
\mathcal{A}=\left\{u\in H_{0}^{1}(\Omega): \|u\|^{2}<\eta \int_{\Omega}u^{2}dx\right\}.
$$
Thus $I(tu)=\frac{I(tu)}{t^{2}}t^{2}$ is positive for $t$ small and
$$
\lim_{t\to +\infty}I(tu)=\lim_{t\to +\infty}\frac{I(tu)}{t^{2}}t^{2}=-\infty.
$$
This fact (jointly with $(f2)$) shows us that ($A_2$) holds for all $u\in \mathcal{A}$. 

\medskip

On the other hand, if $u\in \mathcal{A}^{c}$, $u\neq 0$, then $\|u\|^{2}\geq \eta\int_{\Omega}u^{2}dx$. Thus, it follows from $(f_{1})$ that
$$
I'(u)u=\|u\|^{2}-\int_{[u\neq 0]}\frac{f(u)}{u}u^{2}dx
\geq \int_{[u\neq 0]}\left[\eta-\frac{f(u)}{u}\right]u^{2}dx>0.
$$
This shows that $u$ is not a critical point of $I$.

\medskip

\item Problems involving even terms:

\begin{equation}
\left \{ \begin{array}{ll}
   -\Delta u=\lambda u-u^{2} & \mbox{in $\Omega$,}\\
  u=0 & \mbox{on $\partial\Omega$,}
\end{array}\right.
\end{equation}
where $\lambda<\lambda_1$ is a real parameter and $\Omega$ is a bounded domain in $\R^3$. The associated energy functional is
$$
I(u)=\frac{1}{2}\int_{\Omega}|\nabla u|^{2}dx-\frac{\lambda}{2}\int_{\Omega}u^{2}dx+\frac{1}{3}\int_{\Omega}u^{3}dx.
$$
It can be observed that ($A_2$) holds only for $u$ belonging to
$$
\mathcal{A}=\left\{u\in H_{0}^{1}(\Omega): \int_{\Omega}u^{3}dx<0\right\}.
$$

\end{enumerate}

\medskip

In the previous situations, it is not initially clear how to provide a unified abstract framework that would allow us to guarantee the existence of a homeomorphism between the Nehari set and a proper subset of the unit sphere of $X$. Additionally, another delicate issue would be to find general conditions on $I$ to ensure that such a subset has a differentiable structure, which would allow us to maintain the practical spirit of the approach in \cite{SW}.  

\medskip

Motivated by applications, we develop an abstract boundary-barrier formulation extending the full-sphere setting in \cite{SW}, which is recovered by taking $\mathcal{A} = X \setminus \{0\}$, where the set $\mathcal{A}$ is an open cone in $X$. Under appropriate conditions, we prove the existence of a homeomorphism between $\mathcal{N}$ and $\mathcal{A} \cap \mathcal{S} \neq \emptyset$, which constitutes a proper, possibly non-complete subset of the sphere $\mathcal{S}$ that nevertheless retains a differentiable structure (see Proposition 2.3). As our main abstract result, we establish a minimax theorem (see Theorem \ref{level}), which guarantees the existence of ground-state solutions for this general setting. As a consequence, we provide a unified approach to obtain ground-state solutions for a broad class of elliptic problems and derive a radial characterization of the ground-state energy level in terms of the $C^{1}$-manifold $\mathcal{S}_{\mathcal{A}}$.

\medskip

We would like to emphasize that the abstract result developed here, which deals with the Nehari-type manifold in cones, also presents an additional advantage. Indeed, this approach not only allows us to obtain the existence of ground state solutions for the problem under consideration, together with a radial minimax characterization of its energy level, but it also provides certain qualitative information about the solutions in specific contexts. This feature can be particularly valuable depending on the nature of the problem being studied. In fact, since the solutions obtained lie in the open cone involved in the method, one can derive qualitative properties directly from this geometric constraint. For instance, in the first application presented in Section 3, this manifests as the additional information that the obtained solution is either nonpositive and nonzero or sign-changing. It is worth noting that in this case, since the underlying operator is of integro-differential type and the maximum principle generally fails to hold, any information regarding the sign of the solution is typically delicate to obtain; see, for instance, \cite{AB} and \cite{BPS}. Therefore, the ability of the method to yield such information is a remarkable advantage in this class of problems. In the second application, this qualitative feature translates into the fact that the solution, by belonging to the open cone, has a nontrivial positive part on a subset of $\Omega_0$ of positive measure. Hence, this methodological property can serve as a powerful and insightful tool to extract additional information beyond ``mere existence'' of solutions.

\medskip

To the best of our knowledge, there are few works in the currently available literature that focus on the existence of ground-state solutions for functionals exhibiting this type of behavior in open cones. In \cite{FQS}, for instance, the authors study the existence of ground-state solutions for functionals whose fibering maps exhibit different behaviors in each direction. In particular, they also treat the case where the fibering maps have a unique global maximum point only for functions belonging to an open cone properly contained within a reflexive Banach space. In this context, some abstract results are provided about the existence of ground-state solutions, involving problems whose energy functional is given by $\Phi=I_0-I$, where $I_0,I\in C^1$ are such that $I$ dominates $I_0$ at infinity and $I_0$ is coercive in the set where $I>0$. The authors provide some applications in order to recover established results and provide new existence results for elliptic problems involving homogeneous and non-homogeneous operators, such as Kirchhoff and p\&q-Laplacian operators, and systems.

\medskip

 Since in \cite{FQS}, the authors are not concerned with a topological characterization of the Nehari set, no homeomorphism between $\mathcal{N}$ and any subset of the reflexive Banach space is established. In fact, the authors use an approach inspired in \cite{LWW}, in order to prove that the minimum on $\mathcal{N}$ is a critical point, thus, such a topological study was not necessary for their purpose. The unique-maximum geometry in \cite{FQS} already implies the radial minimax identity by parametrization of the rays. Thus, the distinction is not the absence of such an identity in that work, but the use here of a norm-topology boundary barrier and ambient Palais--Smale compactness on the Nehari set. These compactness requirements should not be identified with the weak-compactness assumptions in \cite{FQS}. We would also like to emphasize that, unlike the approach taken in \cite{FQS}, the arguments used in this work for the proof of Theorem \ref{level} do not require the reflexivity of the Banach space involved.

 \medskip

Some particular instances of problems involving the search for ground-state solutions, specifically in cases where the Nehari set is not a $C^1$ manifold and is contained within an open cone, have been considered in the literature (see, for instance, \cite{FS} and \cite{NPS}). It is important to point out, however, that these works do not propose an abstract framework for addressing a general class of problems under such structural conditions.

\medskip

The present paper is divided as follows: In section \ref{abstract setting}, a general version of the Nehari method in cones is introduced, and some abstract results to obtain ground-state solutions with a minimax type energy level given in terms of a subset of the unit sphere are provided. In Section \ref{applica}, a number of elliptic equations are considered in order to illustrate different relevant situations in which the associated energy functional has a ``geometry'' that fits the abstract result presented in section \ref{abstract setting}. Our main theorems in this section provide ground-state results related to the settings considered in \cite{BPS}, \cite{FQS}, \cite{LZ}, \cite{NPS} and \cite{PZ}.

\section{The Nehari method in open cones}\label{abstract setting}

To describe the Nehari method in cones in an abstract setting, we generalize the ideas introduced in \cite{SW}. Let $(X, \Vert \cdot \Vert)$ be a real Banach space, $\mathcal{S}:=\{u\in X:\|u\|=1\}$ the unit sphere in $X$. For a subset $A\subset X$ and $s\in\mathbb{R}$, we define $s A:=\{su: u\in A\}$. In this paper, by a \textit{cone} in $X$ we mean a nonempty subset $\mathcal{A}$ of $X$ such that $s\mathcal{A}\subset\mathcal{A}$, for all $s>0$. More precisely, since $\mathcal{A}=s[(1/s)\mathcal{A}]\subset s\mathcal{A}$, if $\mathcal{A}$ is a cone then $s\mathcal{A}=\mathcal{A}$, for all $s>0$. Let $I \in C^{1}(X, \R)$ be a functional and assume without loss of generality that $I(0)=0$. A function $\varphi \in C(\R^+, \R^+)$ is said to be a \textit{normalization function} if $\varphi(0)=0$, $\varphi$ is strictly increasing and $\varphi(s) \rightarrow +\infty$, as $s \rightarrow +\infty$. Assume that:

\medskip

\begin{itemize}
\item [$(A_1)$] \ There exists a normalization function $\varphi$ such that      
\[
\psi (u) := \int^{\left\|u\right\|}_{0}\varphi(s) \ ds \ \in \ C^1(X\backslash\left\{0\right\}, \R),
\]
and $J:=\psi^\prime$ is bounded on bounded sets and $J(w)w = 1$, for all $w \in \mathcal{S}$.

\item [$(A_2)$] \ There exists an open cone $\mathcal{A}$ contained in $X\backslash\left\{0\right\}$ with the property that, for each $w \in \mathcal{A}$, there exists a unique $s_w \in (0, +\infty)$ such that if $\alpha_w(s)=I(sw)$, then $\alpha'_w(s) > 0$ for $s \in (0, s_w)$ and $\alpha'_w(s) < 0$, for all $s>s_w$. Furthermore, if $w \in X\backslash\left\{0\right\}$ and $\alpha'_w(s) = 0$ for some $s > 0$, then $w \in \mathcal{A}$ and $s=s_w$.    
\item [$(A_3)$] \ There exists $\delta > 0$ such that 
$\delta \leq s_w$ for all $w \in \mathcal{S}_\mathcal{A}:=\mathcal{S}\cap\mathcal{A}\neq \emptyset$ and for each compact subset $\mathcal{K} \subset \mathcal{S}_\mathcal{A}$ there exists a constant $C_\mathcal{K} > 0$ such that $s_w \leq C_{\mathcal{K}}$ for all $w \in \mathcal{K}$.
\end{itemize}

\medskip

\begin{remark}\label{rema1} 
It is well known that under hypothesis $(A_1)$ the sphere $\mathcal{S}$ is a $C^1$-submanifold of $X$ and the tangent space of $\mathcal{S}$ at $w\in \mathcal{S}$ is given by $
T_w(\mathcal{S})=\left\{z \in X: J(w)z = 0\right\}$. In particular, for any $w \in \mathcal{S}$, there holds the decomposition $X = T_w{(\mathcal{S})} \oplus \R w$. Furthermore, in hypotheses $(A_2)$ and $(A_3)$, if $\mathcal{A}=X\backslash\{0\}$ then $\mathcal{S}_{\mathcal{A}}=\mathcal{S}$ and we are in the same conditions required in \cite{SW}. However, when $\mathcal{A}$ is a proper subset of $X\backslash\{0\}$, then $\mathcal{S}_{\mathcal{A}}$ is a proper open subset of $\mathcal{S}$. Throughout, $\partial\mathcal{S}_{\mathcal{A}}$ denotes its boundary relative to $\mathcal{S}$, closures are taken in $X$, and $\inf\emptyset=+\infty$. The relative boundary may be empty if the sphere is disconnected.
\end{remark}

\medskip

In what follows $X$ satisfies hypothesis $(A_1)$ and the Nehari set $\mathcal{N}$
associated with $I$ is given by
$$
\mathcal{N} := \left\{u \in X\backslash\left\{0\right\}: I'(u)u = 0\right\}.$$

\begin{lemma}         \label{Lem 1.2}
Let $X$ be a real Banach space satisfying $(A_1)$ and suppose that $I$ satisfies $(A_2)-(A_3)$.   
\par\noindent $(a)$ \ If $sw \in \mathcal{N}$ with $s>0$ and $w\in X\backslash \{0\}$, then $s=s_{w}$ and $w\in\mathcal{A}$. In particular, $\mathcal{N}\subset\mathcal{A}$ and if $w\in\mathcal{N}$ then $s_w=1$;
\par\noindent $(b)$ \ $s_{(tw)} = (1/t)s_w$, for each $w \in \mathcal{A}$ and $t > 0$;
\par\noindent $(c)$ \ $\emptyset\neq\mathcal{N} = \left\{s_w w: w \in \mathcal{S}_{\mathcal{A}}\right\}$
\par\noindent $(d)$ \ $\mathcal{N}$ is a closed subset in $X$ and $\mathrm{dist}(\mathcal{N}, 0) > 0$.
\end{lemma}
\begin{proof} 
 
 \noindent $(a)$ Let $w \in X\backslash\{0\}$ and $s > 0$ such that $sw \in \mathcal{N}$. Then $I'(sw)sw =s\alpha'_w(s)=0$, thus $\alpha'_w(s) = 0$. Hypothesis $(A_2)$ implies $w \in \mathcal{A}$ and $s = s_w$.

   \noindent $(b)$ Let $w \in \mathcal{A}$ and $t > 0$. Using the definition of $\alpha_w$ and $(A_2)$, it follows that $0 = \alpha'_{tw}(s_{(tw)}) = t I'((s_{(tw)}t)w)w = t\alpha'_w(s_{(tw)}t)$
and thus $s_{(tw)}t$ is a critical point of $\alpha_w$. Since $s_w$ is the unique critical point of $\alpha_w$, using $(A_2)$ again we obtain $s_{(tw)}t = s_w$, which implies $s_{(tw)} = (1/t)s_w$. 
 
 \noindent $(c)$ It follows from items $(a)$ and $(b)$ that 
\[\mathcal{N} \subset \left\{s_w w: w \in \mathcal{A}\right\} = \left\{s_{(w /\| w\|)}\frac{w}{\Vert w\Vert}: w \in \mathcal{A}\right\} = \left\{s_w w: w \in \mathcal{S}_{\mathcal{A}}\right\}.\]
To prove that $\left\{s_w w: w \in \mathcal{S}_{\mathcal{A}}\right\} \subset \mathcal{N}$, it suffices to note that from $(A_2)$, we have $\alpha'_w(s_w) = 0$ and then $I'(s_w w)w=0$, and hence $I'(s_w w)(s_w w)=0$. Therefore $s_w w \in \mathcal{N}$, for all $w \in \mathcal{S}_{\mathcal{A}}$.

\noindent$(d)$ Let $u \in \mathcal{N}$, from item $(c)$, it follows that $$\mathrm{dist}(\mathcal{N}, 0) =\displaystyle \inf_{u \in \mathcal{S}_{\mathcal{A}}}s_u.
$$ 
Combining this fact with $(A_3)$ we get $\delta > 0$ such that $s_u\geq\delta$ for all $u\in S_{\mathcal{A}}$, that is, $\mathrm{dist}(\mathcal{N}, 0)\geq\delta > 0$. To prove that $\mathcal{N}$ is closed, let $\{u_n\} \subset \mathcal{N}$ be a sequence and $u \in X$ such that $u_n \rightarrow u$ in $X$. Since $\mathrm{dist}(\mathcal{N}, 0)> 0$, we have $u\in X\backslash\{0\}$. On the other hand, by $I \in C^1{(X, \R)}$, we conclude that $0 = I'(u_n)u_n \rightarrow I'(u)u$ as $n$ tends to infinity. Thus $I'(u)u = 0$ with $u\neq 0$, this implies $u \in \mathcal{N}$.
\end{proof}

\begin{remark}\label{bounded}
It is a straightforward consequence of $(A_2)$ and Lemma \ref{Lem 1.2}(c) that $I$ is bounded from below on $\mathcal{N}$.  
\end{remark}

Define the mappings $\widehat{m}: \mathcal{A} \rightarrow \mathcal{N}$ and $m: \mathcal{S}_{\mathcal{A}}\rightarrow \mathcal{N}$ by setting 
\[\widehat{m}(w)=s_w w~~\mbox{and}~ ~m:=\widehat{m}|_{\mathcal{S}_\mathcal{A}}.\]

\begin{lemma}    \label{Lem 1.3} Let $X$ be a real Banach space satisfying $(A_1)$ and suppose that $I$ satisfies $(A_2)-(A_3)$. If $u \in \mathcal{A}$ then $\widehat{m}(t u) = \widehat{m} (u)$,
for all $t>0$.
\end{lemma}

\begin{proof} Let $u \in \mathcal{A}$ and $t>0$ arbitrary. By definition of $\widehat{m}$ we have $\widehat{m}(u) = s_u u$. Now, Lemma \ref{Lem 1.2}{$(b)$} implies $\widehat{m}(tu) = s_{(tu)} tu = \frac{s_u}{t}tu = s_u u$. Therefore, $\widehat{m}(tu) = \widehat{m}(u)$, for all $t > 0$.             
\end{proof}

The following result proves that $m$ is a homeomorphism between $\mathcal{S}_\mathcal{A}$ and $\mathcal{N}$, as we will see later on, such a topological characterization guarantees, under suitable conditions, that minimum points of $I$ on the Nehari set are critical points of $I$ on the whole space $X$.
\begin{proposition}\label{homeo} Let $X$ be a real Banach space satisfying $(A_1)$ and suppose that $I$ satisfies $(A_2)-(A_3)$. Then:	                                \label{Prop 1.1}
\par\noindent $(a)$ The mapping $\widehat{m}$ is continuous;
\par\noindent $(b)$ The mapping $m$ is a homeomorphism between $\mathcal{S}_\mathcal{A}$ and $\mathcal{N}$, and the inverse of $m$ is given by $m^{-1}(u) = \|u\|^{-1}u$.
\end{proposition}
\begin{proof}
$(a)$ Let $\left\{u_n\right\} \subset \mathcal{A}$ and $u \in \mathcal{A}$ be such that $u_n \rightarrow u$ in $X$. From Lemma \ref{Lem 1.3}, we may replace $u_n$ and $u$ by $u_n/\|u_n\|$ and $u/\|u\|$, respectively. Relabeling these normalized vectors, we may therefore assume that $\{u_n\}\subset \mathcal{S}_{\mathcal{A}}$ and $u\in\mathcal{S}_{\mathcal{A}}$. Let $\mathcal{K}$ be the compact set given by the closure of $\{u_n\}$ in $\mathcal{S}_\mathcal{A}$. Then $(A_3)$ implies $\delta \leq s_{u_n} \leq c_{\mathcal{K}}$, for all $n \in \N$ and $\delta, c_{\mathcal{K}} > 0$. Passing to a subsequence if necessary, we obtain
$s_{u_n} \rightarrow \bar{s} > 0$. Thus, $\widehat{m}(u_n) \rightarrow \bar{s}u$ as $n \rightarrow +\infty$. Since $\mathcal{N}$ is closed and
$\{\widehat{m}(u_n)\} \subset \mathcal{N}$, we get $\bar{s}u \in \mathcal{N}$. From Lemma \ref{Lem 1.2}{$(a)$}, we have $\bar{s}=s_u$, which implies $\widehat{m}(u)=\bar{s}u$ and then
$\widehat{m}(u_n) \rightarrow \widehat{m}(u)$ in $X$ along this subsequence. Every convergent subsequence of $s_{u_n}$ has the same limit $s_u$; the compact interval bound then implies convergence of the whole sequence.

\medskip

$(b)$ Note that the mapping $m^{-1}(u)=u/\|u\|$ is well defined, since $\mathcal{N}\subset \mathcal{A}$ (see Lemma \ref{Lem 1.2}{$(a)$}) and $s\mathcal{A} = \mathcal{A}$ for all $s > 0$. To prove that the compositions between $m$ and $m^{-1}$ are the identity, it is enough to use again Lemma \ref{Lem 1.2}{$(a)$} and Lemma \ref{Lem 1.2}{$(b)$}.                                                                 
\end{proof}

We shall consider the functionals
\[\widehat{\Psi}: \mathcal{A} \rightarrow \R \ \ \ \ \mbox{and} \ \ \ \ \Psi: \mathcal{S}_\mathcal{A} \rightarrow \R,\]
defined by
\[\widehat{\Psi}(u)=I(\widehat{m}(u)) \ \ \ \ \mbox{and} \ \ \ \ \Psi:= \widehat{\Psi}|_{\mathcal{S}_\mathcal{A}}.\]
Although we cannot guarantee that $\mathcal{N}$ is a $C^1$-manifold, we shall show that $\widehat{\Psi}$ is of class $C^1$ and there is a one-to-one correspondence between critical points of $\Psi$ and nontrivial critical points of $I$.
\begin{proposition}\label{Abstract}                         \label{Prop 1.2}
 Let $X$ be a real Banach space satisfying $(A_1)$. Suppose that $I \in C^{1}(X, \R)$ satisfies $(A_2)$-$(A_3)$. Then:	                     
\par\noindent $(a)$ \ $\widehat{\Psi}\in C^{1}(\mathcal{A}, \R)$ and 
\begin{equation}      	                            \label{eq 1.7}
\widehat{\Psi}'(w)z = \frac{\left\|\widehat{m}(w)\right\|}{\left\|w\right\|}I'(\widehat{m}(w))z,
\end{equation}
for all $w\in\mathcal{A}$ and $z \in X$;
\par\noindent $(b)$ $\Psi \in C^{1}(\mathcal{S}_\mathcal{A}, \R)$ and
\begin{equation}                                    \label{eq 1.8}
\Psi'(w)z = \left\|m(w)\right\|I'(m(w))z,
\end{equation}
for all $w\in \mathcal{S}_\mathcal{A}$ and $z \in T_w (\mathcal{S}_\mathcal{A})$;
\par\noindent $(c)$ If $\left\{w_n\right\}\subset \mathcal{S}_\mathcal{A}$ is a Palais-Smale sequence for $\Psi$, then $\left\{m(w_n)\right\}\subset \mathcal{N}$ is a Palais-Smale sequence for $I$. If $\left\{m(w_n)\right\} \subset \mathcal{N}$ is a bounded Palais-Smale sequence for $I$, then $\left\{m^{-1}(m(w_n))\right\}=\{w_n\}$ is a Palais-Smale sequence for $\Psi$;
\par\noindent $(d)$ \ $w$ is a critical point of $\Psi$ if, and only if, $m(w)$ is a nontrivial critical point of $I$. Moreover, the corresponding critical values of $\Psi$ and $I$ coincide and
$\displaystyle\inf_{\mathcal{S}_\mathcal{A}}\Psi = \displaystyle\inf_{\mathcal{N}}I$.
\end{proposition}
\begin{proof} $(a)$ Let $w\in\mathcal{A}$, $z \in X$ and $t >0$. By definitions of $\widehat{\Psi}$ and $\widehat{m}$, it follows that
\[\widehat{\Psi}(w+tz) - \widehat{\Psi}(w) = I(\widehat{m}(w + tz)) - I(\widehat{m}(w)) = I(s_{w+tz}(w + tz)) - I(s_w (w)).\]
Taking into account that $s_w$ is a maximum point for the fibering map $\alpha_w$, we obtain
\[I(s_{(w+tz)}(w + tz)) - I(s_w(w)) \leq I(s_{(w + tz)}(w + tz)) - I(s_{(w+tz)}(w)).\]
By Mean Value Theorem, there exists $\theta \in (0, 1)$ such that
\begin{equation}				                \label{eq 1.9}
\widehat{\Psi}(w+tz) - \widehat{\Psi}(w) \leq I'(s_{(w + tz)}(w + \theta tz))s_{(w +tz)}tz.
\end{equation}
Analogously, there exists $\zeta \in (0, 1)$ such that
\begin{equation}	                              \label{eq 1.10}
\widehat{\Psi}(w+tz) - \widehat{\Psi}(w) \geq I'(s_{w} (w + \zeta tz))s_{w}tz.
\end{equation}

\medskip

Using the continuity of the mapping $\widehat{m}$, $(\ref{eq 1.9})$ and $(\ref{eq 1.10})$ we obtain
\[\lim_{t \rightarrow 0} \frac{\widehat{\Psi}(w + tz) - \widehat{\Psi}(w)}{t} = I'(s_w w)s_w z.\]
Since $s_w = \left\|\widehat{m}(w)\right\| / \left\|w\right\|$, for all $w\in\mathcal{A}$, we obtain $(\ref{eq 1.7})$.

\medskip

To justify Fr\'echet, rather than only directional, differentiability, apply the same two maximality inequalities with an arbitrary increment $h\in X$ tending to zero. Since $s_{w+h}\to s_w$ and $I'$ is norm-continuous, both bounding expressions equal $s_wI'(s_ww)h+o(\|h\|)$. Thus \eqref{eq 1.7} is the Fr\'echet derivative. The map $w\mapsto s_wI'(s_ww)\in X'$ is continuous, so $\widehat\Psi\in C^1(\mathcal A,\mathbb R)$.     

\medskip

$(b)$ For $w\in \mathcal{S}_\mathcal{A}$ and $z \in T_w (\mathcal{S}_\mathcal{A})$, let $\gamma: [0, 1] \rightarrow \mathcal{S}_\mathcal{A}$ be a differentiable curve satisfying $\gamma(0)=w$ and $\gamma'(0) = z \in T_{w}(\mathcal{S}_\mathcal{A})$. Then, by definition
\[\Psi'(w)z = \frac{d}{dt}\widehat{\Psi}(\gamma(t))|_{t=0}.\]

As a consequence of the Chain rule and $(a)$, we obtain
\[\Psi'(w)z = \widehat\Psi'(\gamma(0))\gamma'(0) = \left\|m(w)\right\| I'(m(w))z.\]

\medskip

$(c)$ Hypothesis $(A_1)$ implies $X = T_{w}(\mathcal{S}_\mathcal{A}) \oplus \R w$, for each $w \in \mathcal{S}_\mathcal{A}$.  $P: X \rightarrow T_{w}(\mathcal{S}_\mathcal{A})$, defined as $P(v + tw) = v$, let us show that there exists $c > 0$ such that
\begin{equation}\label{treta}
\left\|P(v + tw)\right\| \leq c \left\|v + tw\right\|, \ \ \forall \ w \in \mathcal{S}_\mathcal{A}, \ \forall v \in T_{w}(\mathcal{S}_\mathcal{A}) \ \ \mbox{and} \ t \in \R.
\end{equation}
By $(A_1)$, we know that $T_{w}(\mathcal{S}_\mathcal{A})=Ker\{J(w)\}$ and $J(w)w=1$. Thus, 
\[J(w)(v +tw) = J(w)v + tJ(w)w = t.\]
If $\left\|v + t w\right\| = 1$, Once $J$ is bounded on bounded sets, we get
\[\left|t\right|=|J(w)(v+tw)|\leq\left\|J(w)\right\|_{X'}\left\|v+tw\right\|\leq C,\]
for some constant $C > 0$. Therefore,
\begin{equation}\label{treta2}
\left\|P(v + tw)\right\| = \left\|v + tw - tw\right\| \leq  \left\|v + tw\right\|+ \left|t\right| \leq c,
\end{equation}
for all $w \in \mathcal{S}_\mathcal{A}$, $v \in T_{w}(\mathcal{S}_\mathcal{A})$ and $t \in \R$ such that $\left\|v + t w\right\| = 1$, where $c:=1+C$. It follows from \eqref{treta2} that \eqref{treta} holds. Furthermore, by  $(b)$, we have
\[\left\|\Psi'(w)\right\|_{T_{w}(\mathcal{S}_\mathcal{A})'} = \sup_{{z \in T_{w}(\mathcal{S_\mathcal{A}})}, {\left\|z\right\| = 1}} \Psi'(w)z = \left\|u\right\| \sup_{{z \in T_{w}(\mathcal{S}_\mathcal{A})}, {\left\|z\right\| = 1}}I'(u)z\]
with $u = m(w)$. Thus,
$$
    \left\|\Psi'(w)\right\|_{T_{w}(\mathcal{S}_\mathcal{A})'}\leq  \|u\|\|I'(u)\|_{X'}= \left\|u\right\| \sup_{z+tw \in X\backslash{\left\{0\right\}}}\frac{|I'(u)(z + tw)|}{\left\|z + tw\right\|}.
    $$
    Now, by \eqref{treta} and
\[I'(u)w = I'(u)\left(\frac{u}{\left\|u\right\|}\right) = 0,\]
we obtain
\begin{equation}                                  \label{eq 1.12}
    \left\|\Psi'(w)\right\|_{T_{w}(\mathcal{S}_\mathcal{A})'}\leq  \|u\|\|I'(u)\|_{X'} \leq c\left\|u\right\|\sup_{{z  \in T_{w}(\mathcal{S}_\mathcal{A})\backslash{\left\{0\right\}}}, \left\|z\right\|=1}\frac{|I'(u)z|}{\left\|z\right\|} = c\left\|\Psi'(w)\right\|_{T_{w}(\mathcal{S}_\mathcal{A})'}.
\end{equation}
Since $u=m(w)\in\mathcal N$, Lemma \ref{Lem 1.2} gives $\|u\|\geq\delta$. The second inequality in \eqref{eq 1.12} therefore transfers the vanishing derivative of $\Psi$ to that of $I$. For the converse, the first inequality is used together with the boundedness of $m(w_n)$. The energy convergence is identical by definition of $\Psi$.

\medskip

$(d)$ The proof of the first part of this item follows from the inequality in \eqref{eq 1.12}. The second part is a consequence of the definition of $\Psi$. 
\end{proof}

In the following, we provide an abstract result for the existence of ground state critical points and a minimax characterization for the infimum $c_\mathcal{N}:=\inf_{u \in {\mathcal{N}}}I(u)$, in the case that $\mathcal{N}$ is homeomorphic to $\mathcal{S}_{\mathcal{A}}$. The central feature is the presence of hypothesis $(H)$, which separates the minimizing energy from a lower semicontinuous barrier on the relative boundary. It allows Ekeland's principle to be applied on the complete closure of $\mathcal S_{\mathcal A}$, and forces the resulting variational points to remain in its interior. It does not assert a uniform positive distance from the boundary for every minimizing sequence.

Here compactness is needed only for ambient Palais--Smale sequences lying in the Nehari set: we write $(PS)^{\mathcal N}_c$ when every sequence $u_n\in\mathcal N$ with $I(u_n)\to c$ and $\|I'(u_n)\|_{X'}\to0$ has a subsequence converging in $X$. This refers to the derivative of $I$ on $X$, not to a derivative of $I|_{\mathcal N}$. The usual $(PS)_c$ condition implies this property.

\begin{theorem}\label{level} Let $X$ be a real Banach space satisfying $(A_1)$, $I:X\to \R$ a $C^{1}$-functional satisfying $(A_2)-(A_3)$, the $(PS)^{\mathcal N}_{c_{\mathcal{N}}}$-condition and such that: 

\medskip

\begin{itemize}
\item[$(H)$] There exists a lower semicontinuous map, $\Gamma:\partial \mathcal{S}_{\mathcal{A}}\to\R\cup\{+\infty\}$, verifying $c_{\mathcal{N}}<\displaystyle\inf_{u\in \partial \mathcal{S}_{\mathcal{A}}}\Gamma(u)$ and such that, if $\{u_n\}\subset\mathcal{S}_{\mathcal{A}}$ and $u_n\rightarrow u$ in $X$ with $u\in \partial \mathcal{S}_{\mathcal{A}}$, then 
$$
\liminf_{n\to\infty}{I(m(u_n))}\geq\Gamma(u).
$$ 
\end{itemize}
Then, there exists a nontrivial critical point of $I$ at the strictly positive level 
\[c_{\mathcal{N}} = \inf_{w \in {\mathcal{A}}}\max_{s>0} I(sw) = \inf_{w \in \mathcal{S}_\mathcal{A}}\max_{s>0} I(sw).\]
\end{theorem}
\begin{proof}
It follows from Remark \ref{bounded} and definition of $c_{\mathcal{N}}$ that there exists a sequence $\{u_{n}\}\subset\mathcal{N}$ such that $I(u_{n})\to c_{\mathcal{N}}$. By $(A_2)-(A_3)$ and Proposition \ref{homeo}, we conclude that $v_n:=m^{-1}(u_n) \in \mathcal{S}_\mathcal{A}$ and $\Psi(v_n) \rightarrow c_\mathcal{N}$. Let us prove that there exists a $(PS)_{c_{\mathcal{N}}}$-sequence for $\Psi$. To this end, we will use the Ekeland's variational principle, see \cite{Eke}. Let $M:=\overline{\mathcal{S}_\mathcal{A}}$ be the complete metric space with the metric induced by the norm of $X$, and consider the functional $\Upsilon: M \rightarrow \R \cup \left\{+\infty\right\}$ given by
\begin{equation}	             \label{eq 2.30}
\Upsilon(v) = \left\{
\begin{array}{rcl}
\Psi(v),  &\mbox{if}& \ v \in \mathcal{S}_\mathcal{A},\\[0.2cm]
\Gamma(v),  &\mbox{if} &\ v \in \partial\mathcal{S}_\mathcal{A}.
\end{array}
\right.
\end{equation}

{\bf Claim 1:} $\Upsilon$ is lower semicontinuous. 

\medskip

In fact, let $\left\{v_n\right\} \subset M$ and $v \in M$ be such that $v_n \rightarrow v$ in $M$ as $n\to\infty$. If $v \in \mathcal{S}_\mathcal{A}$, then, for $n$ large,  
$$
\Upsilon(v_n) = \Psi(v_n) \rightarrow \Psi(v) = \Upsilon(v), 
$$
because $\Psi$ is continuous, see Proposition \ref{Abstract}. On the other hand, if $v \in \partial\mathcal{S}_\mathcal{A}$, first take a subsequence realizing $\liminf\Upsilon(v_n)$, and then, if necessary, a further subsequence whose terms all lie in the interior or all lie in the boundary. {\bf Case 1:} This subsequence lies in $\mathcal S_{\mathcal A}$. By hypothesis $(H)$, it follows that
\[\Upsilon(v)=\Gamma(v) \leq \liminf_{n \rightarrow +\infty} \Psi(v_n) = \liminf_{n \rightarrow +\infty}\Upsilon(v_n).\]
{\bf Case 2:} This subsequence lies in $\partial\mathcal S_{\mathcal A}$. In this case,
\[
\Upsilon(v) = \Gamma(v) \leq \liminf_{n \rightarrow +\infty}\Gamma(v_n)= \liminf_{n \rightarrow +\infty}\Upsilon(v_n),
\]
since, by hypothesis $(H)$, the functional $\Gamma$ is lower semicontinuous. 

It is a consequence of hypothesis $(H)$ and the lower boundedness of $I$ on $\mathcal{N}$ that $\Upsilon$ is bounded from below and $\inf_{M} \Upsilon = c_{\mathcal{N}}$.
Thus, by Ekeland's variational principle, for $\epsilon, \delta > 0$ small enough and $u \in \Upsilon^{-1}[c_{\mathcal{N}}, c_{\mathcal{N}} +\epsilon]$, there exists $v \in M$ such that
\begin{equation}                      \label{eq 101}
    c_{\mathcal{N}} \leq \Upsilon(v) \leq \Upsilon(u), \ \  \left\|u - v\right\| \leq \delta \ \mbox{and} \ \Upsilon(w) > \Upsilon(v) - (\epsilon/\delta)\left\|v - w\right\|, \ \ \ \forall \ w \neq v.
\end{equation}                                 
On the other hand, it follows from $(H)$ that, for $0<\epsilon<\inf_{\partial\mathcal{S}_{\mathcal{A}}}\Gamma-c_{\mathcal{N}}$, we have $\Upsilon^{-1}[c_{\mathcal{N}}, c_{\mathcal{N}}+\epsilon]\subset \mathcal{S}_{\mathcal{A}},$
\begin{equation}                                      \label{eq 100}
    \Upsilon^{-1}[c_{\mathcal{N}}, c_{\mathcal{N}}+\epsilon] = \Psi^{-1}[c_{\mathcal{N}}, c_{\mathcal{N}}+\epsilon] \ \ \mbox{and} \ \ \Upsilon(v) = \Psi(v).
\end{equation}
   Choosing the initial minimizing sequence so that $\Psi(v_n)\leq c_{\mathcal N}+1/n^2$, we can take $\epsilon = 1/n^2$, $\delta = 1/n$ and $u=v_n$, to obtain $v=\hat{v}_n \in \mathcal{S}_\mathcal{A}$ satisfying (by \eqref{eq 101})
\begin{equation}                                      \label{eq 233}
    \Psi(\hat{v}_n) \rightarrow c_\mathcal{N},
\end{equation}
\begin{equation}                                      \label{eq 234}
\left\|v_n - \hat{v}_n\right\| \rightarrow 0    
\end{equation}
and
\begin{equation}                                      \label{eq 235}
    \Upsilon(w) > \Upsilon(\hat{v}_n) - \left(1/n\right)\left\|\hat{v}_n - w\right\|, \ \ \ \forall \ w \neq \hat{v}_n.
\end{equation}

Let $\gamma_n: (-\delta_n, \delta_n) \rightarrow \mathcal{S}_\mathcal{A}$ be a differentiable curve, with $\delta_n > 0$ small enough, such that
\begin{equation}                                       \label{eq 236}
    \gamma_n(0) = \hat{v}_n \ \ \mbox{and} \ \ \gamma'_n(0) = z \in T_{\hat{v}_n}(\mathcal{S}_\mathcal{A})\setminus \{0\},
    \end{equation}
where $T_{\hat{v}_n}(\mathcal{S}_\mathcal{A})$ denotes the tangent space of $\mathcal{S}_\mathcal{A}$ at $\hat{v}_n$. Considering $w = \gamma_n(t)$, with $t\neq 0$, it follows from $(\ref{eq 235})$ and $(\ref{eq 236})$ that
\begin{equation}                                        \label{eq 237}
    \Upsilon(\gamma_n(t)) > \Upsilon(\hat{v}_n) - \left(1/n\right)\left\|\gamma_n(0) - \gamma_n(t)\right\|.
\end{equation}
It follows from definition of $\Upsilon$, that
\begin{equation}                                        \label{eq 238}
    -[\Psi(\gamma_n(t)) - \Psi(\gamma_n(0))] < (1/n)\left\|\gamma_n(0) - \gamma_n(t)\right\|.
\end{equation}
Multiplying both sides of $(\ref{eq 238})$ by $1/t$, passing to the limit as $t \rightarrow 0^+$ and using $(\ref{eq 236})$, we obtain 
\[-\Psi'(\hat{v}_n)z \leq (1/n)\left\|z\right\|,\]
where $z \in T_{\hat{v}_n}(\mathcal{S}_\mathcal{A})$ is arbitrary. Since $\Psi'$ is a linear and continuous functional, we have $\left|\Psi'(\hat{v}_n)z\right| \leq (1/n)\left\|z\right\|$.
Thus, 
\begin{equation}\label{okk}
\left\|\Psi'(\hat{v}_n)\right\|_{T_{\hat{v}_n}(\mathcal{S}_\mathcal{A})'} \rightarrow 0 \ \mbox{as $n \rightarrow +\infty$}.
\end{equation}
By $(\ref{eq 233})$ and $(\ref{okk})$, we conclude that $\left\{\hat v_n\right\}$ is a Palais-Smale sequence for $\Psi$ at the level $c_{\mathcal{N}}$.

\medskip

Set $\hat u_n:=m(\hat v_n)$. Proposition \ref{Abstract} shows that $\{\hat u_n\}\subset\mathcal N$ is an ambient $(PS)_{c_{\mathcal N}}$ sequence for $I$. The assumed $(PS)^{\mathcal N}_{c_{\mathcal N}}$ condition therefore provides, up to a subsequence, $\hat u_n\to u$ in $X$. The closedness and positive distance from zero in Lemma \ref{Lem 1.2}(d) give $u\in\mathcal N$. By continuity,
\[I'(u)=0\qquad\mbox{and}\qquad I(u)=c_{\mathcal N}.\]
Thus $u$ is nontrivial. Every nonzero critical point belongs to $\mathcal N$, so it is a ground-state critical point.

\medskip

It remains to prove the minimax characterization of $c_{\mathcal{N}}$. Since $\widehat{m}(w) = \widehat{m}(w/\left\|w\right\|)$, for all $w \in \mathcal{A}$, we have
\begin{equation}					 \label{eq 1.13}
\inf_{w \in \mathcal{A}}\max_{s > 0} I(sw) = \inf_{w \in {\mathcal{S}_\mathcal{A}}}\max_{s > 0}I(sw).
\end{equation}	                                                      
The definition of $\Psi$ implies
\begin{equation}			        \label{eq 1.14}
\inf_{w \in {\mathcal{S}_\mathcal{A}}}\max_{s > 0}I(sw) = \inf_{w \in \mathcal{S}_\mathcal{A}} \Psi(w).
\end{equation}
It is a consequence of Proposition \ref{Abstract}(d), $(\ref{eq 1.13})$ and $(\ref{eq 1.14})$, that
\[c_{\mathcal{N}} = \inf_{w \in \mathcal{S}_\mathcal{A}} \Psi(w)= \inf_{w \in \mathcal{S}_{\mathcal{A}}}\max_{s > 0}I(sw)=\inf_{w \in \mathcal{A}}\max_{s > 0} I(sw).\] Finally, since the nonzero minimizer $u$ is the positive fibering maximum on its ray and $I(0)=0$, assumption $(A_2)$ gives $c_\mathcal N=I(u)>0$.                             
\end{proof}

\begin{remark}
If $\mathcal{A}=X\backslash\{0\}$ in Theorem \ref{level} then the relative boundary is empty and assumption $(H)$ is omitted. Assumption $(A_2)$ remains necessary. The proof applies with $\Upsilon=\Psi$ on the complete sphere.
\end{remark}

To finish this section we would like to point out that a symmetric counterpart of Theorem \ref{level}, ensuring the existence of multiple critical points, can indeed be derived from the ideas developed herein. However, pursuing such an extension lies beyond the scope of the present work.

\section{Applications}\label{applica}

The main results in subsections \ref{31} and \ref{32}, that is, Theorems \ref{dito} and \ref{Teo 4.1}, complement the study made in \cite{BPS} by considering other classes of reactions terms.

All six applications involve proper spherical cones with nonempty relative boundary, so the full-sphere formulation of \cite{SW} does not apply directly. A nonempty boundary alone, however, does not distinguish our method from \cite[Section~3.4]{SW}. The first two applications and the Kirchhoff problem use infinite barriers; the asymptotic problem allows finite barriers, Subsection~\ref{finite-cubic} realizes a finite boundary limit through a different mechanism, and Subsection~\ref{gradient-barrier} retains the eigenfunction-gradient profile in its finite threshold. These last two applications require the finite-barrier alternative in $(H)$.

\subsection{An integro-differential problem: Perturbation of the logistic function}\label{31}
In this section, we are interested in studying the existence of a ground-state solution for the following problem
\begin{equation}	        \label{Prob P_2}\tag{$P_1$}
\left\{
\begin{array}{rcl}
-\Delta u + a \displaystyle\int_{\Omega} u \ dx &=& \lambda u - u^2 + g(u) \ \mbox{in} \ \Omega,\\[0.4cm]
u &=& 0 \ \mbox{on} \ \partial \Omega,
\end{array}
\right.
\end{equation}
where $\Omega \subset \R^N$ is a bounded domain with smooth boundary, $1 \leq N < 6$, $a > 0$ and $\lambda < \lambda_1(a)$. Set $\kappa:=\lambda_1(a)-\lambda>0$, where $\lambda_1(a)$ is the first eigenvalue associated to the operator given in the left hand side of the first equation in \eqref{Prob P_2} with the Dirichlet boundary condition (see \cite{BPS} for details about the existence of eigenvalues for this class of problems).  Furthermore, we assume that $g \in C(\R)$ satisfies:
\begin{itemize}
    \item[$(g_1)$]  $\displaystyle\lim_{s \rightarrow 0} \frac{g(s)}{s} = 0 \ \ \ \mbox{and} \ \ \ \displaystyle\lim_{\vert s\vert \rightarrow +\infty} \frac{g(s)}{s} = g_{\infty}\in[0,\kappa)$;

    \item[$(g_2)$] The map $s \mapsto g(s)/\vert s\vert$ is nondecreasing on $\mathbb R\setminus\{0\}$.
\end{itemize}

\medskip

Problems involving an integro-differential operator like in \eqref{Prob P_2} have attracted a lot of attention in recent years since the seminal papers \cite{AB, ASHR} (see also \cite{Cat, PF}). We recommend the following references to illustrate the presence of the present operator as a model for situations in different fields of science: \cite{ASHR, FP, Cha}. Some other recent papers related to this class of problems are \cite{ACS, CLL, dLd}. 

\medskip

Associated with the problem \eqref{Prob P_2}, we have the energy functional $I_a: H^{1}_0(\Omega) \rightarrow \R$, given by 
\begin{equation}                     \label{eq 3.1}
I_a(u) = \frac{1}{2}\int_{\Omega}\left|\nabla u\right|^2 \ dx + \frac{a}{2} \left(\int_{\Omega} u \ dx\right)^2 - \frac{\lambda}{2}\int_{\Omega} u^2 \ dx + \frac{1}{3}\int_{\Omega}u^3 \ dx - \int_{\Omega} G(u) \ dx,
\end{equation}
where $G(s) = \int^s_{0} g(\tau) \ d\tau$. It can be shown that $I_a \in C^1(H^{1}_{0}(\Omega), \R)$ with 
\[I_a'(u)v = \int_{\Omega}\nabla u\nabla v \ dx + a \int_{\Omega} u \ dx \int_{\Omega} v \ dx - \lambda \int_{\Omega} u v \ dx + \int_{\Omega} u^2 v \ dx - \int_{\Omega} g(u) v \ dx, \ \ \ \forall \ v \in H^{1}_{0}(\Omega).\]
For each $a \in (0, (\lambda_2 - \lambda_1)/\vert \Omega\vert)$, where $\lambda < \lambda_1(a)$ and $\lambda_1$ and $\lambda_2$ are the first and the second eigenvalue of $(-\Delta,\, H_0^1(\Omega))$, respectively. The norm
\[\left\|u\right\|_a^2 := \int_{\Omega}\vert\nabla u\vert^2 \ dx + a \left(\int_{\Omega}u \ dx\right)^2 - \lambda\int_{\Omega} u^2 \ dx,\]
whose associated inner product is
$$
\langle u, v\rangle_a:= \int_{\Omega}\nabla u\nabla v \ dx + a \int_{\Omega}u \ dx\int_{\Omega}v \ dx - \lambda\int_{\Omega} u v \ dx,
$$
is equivalent to the usual norm of $H^{1}_{0}(\Omega)$, see \cite{BPS}. More precisely, with
\[
B_a(u,u)=\int_\Omega|\nabla u|^2dx+a\left(\int_\Omega u\,dx\right)^2,
\qquad \lambda_1(a)=\inf_{u\neq0}\frac{B_a(u,u)}{|u|_2^2},
\]
the shifted norm satisfies
\begin{equation}\label{shifted-spectral}
\|u\|_a^2\geq\kappa |u|_2^2,\qquad \kappa=\lambda_1(a)-\lambda>0.
\end{equation}
The Nehari set associated with the functional $I_a$ is defined by
\[\mathcal{N}_a = \left\{u \in H^{1}_{0}(\Omega)\backslash\left\{0\right\}: \Vert u\Vert_a^2 = -\int_{\Omega} u^3 \ dx + \int_{\Omega}g(u)u \ dx\right\}.\]
Let us define
\begin{equation}       \label{eq 3.2}
\mathcal{A} = \left\{u \in H^{1}_{0}(\Omega): \int_{\Omega} u^3 \ dx < 0\right\}
\end{equation}
and $\mathcal{S}_\mathcal{A}:=\mathcal{S}_a\cap\mathcal{A}$, given by
\[\mathcal{S}_\mathcal{A} = \left\{u \in \mathcal{S}_a: \int_{\Omega} u^3 \ dx < 0\right\},\]
where $\mathcal{S}_a := \left\{u \in H^{1}_{0}(\Omega): \Vert u\Vert_a = 1\right\}$. We also denote by $\partial\mathcal{S}_\mathcal{A}$ the relative boundary of the set $\mathcal{S}_\mathcal{A}$ as a topological subspace of $\mathcal{S}_a$, which is given by
\[\partial\mathcal{S}_\mathcal{A} = \left\{u \in \mathcal{S}_a: \int_{\Omega} u^3 \ dx = 0\right\}.\]
Finally, $\mathcal{S}^c_\mathcal{A}$ is the complement of the set $\mathcal{S}_\mathcal{A}$ in $\mathcal{S}_a$, defined by
\[\mathcal{S}^c_\mathcal{A} = \left\{u \in \mathcal{S}_a: \int_{\Omega} u^3 \ dx \geq 0\right\}.\]

The next lemma gives some properties of the set $\mathcal{A}$.
\begin{lemma}                             \label{Lem 3.2}
Suppose that $a \in \left(0, (\lambda_2 - \lambda_1)/\vert \Omega\vert\right)$ and $\lambda < \lambda_1(a)$. Then, $\mathcal{A}$ is a nonempty open set and $s\mathcal{A}\subset \mathcal{A}$, for all $s > 0$.
\end{lemma}
\begin{proof}
The continuous cubic map $h(u)=\int_\Omega u^3dx$ satisfies $h(tu)=t^3h(u)$ for $t>0$. Hence $\mathcal A=h^{-1}(( -\infty,0))$ is an open cone; any nonzero nonpositive test function shows that it is nonempty.
\end{proof} 

\begin{proposition}                           \label{Prop 3.1}
Suppose that $a \in \left(0, (\lambda_2 - \lambda_1)/\vert \Omega\vert\right)$, $\lambda < \lambda_1(a)$ and $g$ satisfies $(g_1)-(g_2)$. Then $(A_1)-(A_3)$ hold.
\end{proposition}
\begin{proof}
Since the shifted space is Hilbert, $(A_1)$ holds with $\varphi(t)=t$, $\psi(u)=\|u\|_a^2/2$ and $J(u)v=\langle u,v\rangle_a$. Hypotheses $(g_1)$--$(g_2)$ give
\[
0\leq g(t)t\leq g_\infty t^2,\qquad 0\leq G(t)\leq\tfrac12g(t)t,
\]
and $s\mapsto g(sw)/(sw)$ is nondecreasing for every nonzero real $w$. For $w\in\mathcal A$,
\begin{equation}\label{eq 3.3}
\alpha_w(s)=\frac{s^2}{2}\|w\|_a^2+\frac{s^3}{3}\int_\Omega w^3dx-\int_\Omega G(sw)dx.
\end{equation}
Its normalized derivative is
\begin{equation}\label{eq 3.5}
\frac{\alpha_w'(s)}s=\|w\|_a^2+s\int_\Omega w^3dx
-\int_{[w\neq0]}\frac{g(sw)}{sw}w^2dx.
\end{equation}
It is strictly decreasing, tends to $\|w\|_a^2>0$ at zero by dominated convergence, and tends to $-\infty$ at infinity. Its unique zero therefore has the signs required in $(A_2)$. If $w\in\mathcal A^c\setminus\{0\}$, the same formula and \eqref{shifted-spectral} give $\alpha_w'(s)/s\geq(1-g_\infty/\kappa)\|w\|_a^2>0$, so no positive critical point exists outside the cone.

For $(A_3)$, $(g_1)$ gives $|g(t)|\leq\varepsilon|t|+C_\varepsilon|t|^2$ for every $\varepsilon>0$. The $L^2$ and $L^3$ embeddings imply
\[
I_a'(u)u\geq(1-\varepsilon/\kappa)\|u\|_a^2-C_\varepsilon'\|u\|_a^3.
\]
Choosing $0<\varepsilon<\kappa$ gives a positive lower bound for $\|u\|_a$ on $\mathcal N_a$, hence for $s_w$ on $\mathcal S_{\mathcal A}$. At a fibering maximum, \eqref{eq 3.5} and $g(t)t\geq0$ give
\[
1=s_w\left(-\int_\Omega w^3dx\right)
+\int_{[w\neq0]}\frac{g(s_ww)}{s_ww}w^2dx
\geq s_w\left(-\int_\Omega w^3dx\right).
\]
For compact $\mathcal K\subset\mathcal S_{\mathcal A}$, the continuous positive function $-\int_\Omega w^3dx$ has a positive minimum on $\mathcal K$. This yields the upper bound for $s_w$ required in $(A_3)$.
\end{proof}

\begin{proposition}        \label{Prop 3.3}
Suppose that $a \in \left(0, (\lambda_2 - \lambda_1)/\vert \Omega\vert\right)$, $\lambda < \lambda_1(a)$ and $g$ satisfies $(g_1)-(g_2)$. Then, the functional $I_a$ is coercive on $\mathcal{N}_a$. 
\end{proposition}
\begin{proof}
By $(g_1)$--$(g_2)$, $0\leq g(t)t\leq g_\infty t^2$ and $0\leq G(t)\leq g(t)t/2$. For $u\in\mathcal N_a$, the Nehari identity and \eqref{shifted-spectral} give
\[
-\int_\Omega u^3dx=\|u\|_a^2-\int_\Omega g(u)u\,dx
\geq(1-g_\infty/\kappa)\|u\|_a^2=:\theta\|u\|_a^2.
\]
Substitution in the energy yields
\begin{equation}\label{eq 3.15}
I_a(u)=-\frac16\int_\Omega u^3dx+\int_\Omega\left[\frac12g(u)u-G(u)\right]dx
\geq\frac\theta6\|u\|_a^2.
\end{equation}
Since $\theta>0$, this proves coercivity on $\mathcal N_a$ directly.
\end{proof}

Now we study the behavior of the energy functional associated with problem $(P_1)$, when the minimizing sequence approaches the boundary of the set $\mathcal{S}_{\mathcal{A}}$. 
\begin{proposition}                    \label{Prop 3.2}
Suppose that $a \in \left(0, (\lambda_2 - \lambda_1)/\vert \Omega\vert\right)$, $\lambda < \lambda_1(a)$ and $g$ satisfies $(g_1)-(g_2)$. If $\left\{w_n\right\} \subset \mathcal{S}_\mathcal{A}$ satisfies $\mathrm{dist}(w_n, \partial\mathcal{S}_\mathcal{A}) \rightarrow 0$, then $s_{w_n} \rightarrow +\infty$ and $\Psi_a(w_n)= I_a(s_nw_n) \rightarrow +\infty$ as $n \rightarrow +\infty$.
\end{proposition}
\begin{proof} 
Let $\left\{w_n\right\} \subset \mathcal{S}_\mathcal{A}$ be such that $\mathrm{dist}(w_n, \partial\mathcal{S}_\mathcal{A}) \rightarrow 0$. Since $\left\{w_n\right\}$ is bounded, there exists $w \in H^{1}_{0}(\Omega)$ such that, up to a subsequence, $w_n \rightharpoonup w$ in $H^{1}_{0}(\Omega)$. Once $\mathrm{\mathrm{dist}}(w_n, \partial\mathcal{S}_\mathcal{A})  \rightarrow 0$, there exists $\left\{z_n\right\} \subset \partial\mathcal{S}_\mathcal{A}$ such that 
\begin{equation}                            \label{eq 3.8}
    \left\|w_n - z_n\right\|_a \rightarrow 0 \ \ \mbox{as} \ n \rightarrow +\infty.
\end{equation}
By H\"older's inequality and the continuous embedding into $L^3$, the unit norms of $w_n,z_n$ give
\[
\begin{aligned}
\left|\int_\Omega w_n^3dx\right|
&=\left|\int_\Omega(w_n-z_n)(w_n^2+w_nz_n+z_n^2)dx\right|\\
&\leq |w_n-z_n|_3(|w_n|_3^2+|w_n|_3|z_n|_3+|z_n|_3^2)
\leq C\|w_n-z_n\|_a.
\end{aligned}
\]
Together with \eqref{eq 3.8}, this yields
\begin{equation}                                           \label{eq 3.9}
    \int_{\Omega} w_n^3 \ dx \rightarrow 0, \ \ \mbox{as} \ n \rightarrow +\infty.
\end{equation}
Using $(\ref{eq 3.9})$ and Sobolev compact embedding, we get
\begin{equation}                        \label{eq 3.10}
   \int_{\Omega} w^3 \ dx = 0, 
\end{equation}
The weak limit satisfies $\|w\|_a\leq1$; the integral identity alone does not imply that it belongs to the unit sphere.

\medskip

We claim that $s_n:=s_{w_n} \rightarrow +\infty$ as $n \rightarrow +\infty$. In fact, suppose by contradiction that, for some subsequence, $\left\{s_n\right\}$ is bounded, then there exists $s_0 > 0$ such that $s_n \rightarrow s_0$, as $n \rightarrow +\infty$. Once
\[s_n = s^2_n\left(-\int_{\Omega} w^{3}_{n} \ dx \right) + \int_{\Omega}g(s_n w_n)w_n \ dx,\]
by Lebesgue dominated convergence theorem, we have
\begin{equation*}                                           
s_0 = s_0^2 \left(-\int_{\Omega} w^3 \ dx \right) + \int_{\Omega} g(s_0 w)w \ dx.    
\end{equation*}
Thus, by (\ref{eq 3.10}), we obtain
\begin{equation}                        \label{eq 3.11}
s_0 = \int_{\Omega}g(s_0w)w \ dx.    
\end{equation}
Now, we have to consider two cases:\\

{\bf Case 1:} $w = 0$. 

Since $s_0>0$, the equation $(\ref{eq 3.11})$ leads us to a contradiction. 

{\bf Case 2:} $w \neq 0$.

In this case, we obtain from \eqref{eq 3.11} that
\[1 = \int_{[w \neq 0]}\frac{g(s_0 w)}{s_0 w}w^2 \ dx.\]
Now, $(g_1)$ and $(g_2)$ imply
\[1 \leq g_{\infty} \int_{\Omega}w^2 \ dx \leq \frac{g_{\infty}}{\kappa},\] 
contradicting the fact that $g_{\infty}<\kappa$. 

\medskip

Therefore, in any case 
\begin{equation}                          \label{eq 3.12}
    s_n \rightarrow +\infty, \ \ \ \mbox{as} \ \ n \rightarrow +\infty.
\end{equation}

Finally, we will show that $\Psi_a(w_n) \rightarrow +\infty$ as $n \rightarrow +\infty$. Note that
\begin{equation}                          \label{eq 3.13}
-s_n^3\int_{\Omega}w^{3}_{n} \ dx + \int_{\Omega}g(s_n w_n)s_nw_n \ dx - s_n^2=0.    
\end{equation}
 
Thus, by \eqref{eq 3.13}, it turns out that
\[
\Psi_a(w_n) = \frac{s_n^3}{6}\left(-\int_{\Omega}w_n^3 \ dx\right) + \int_{\Omega}\left[\frac{1}{2}g(s_n w_n)s_n w_n - G(s_n w_n) \ dx\right].
\]
The Nehari identity \eqref{eq 3.13} and \eqref{shifted-spectral} give
\[
-s_n\int_\Omega w_n^3dx
=1-\frac1{s_n^2}\int_\Omega g(s_nw_n)s_nw_n\,dx
\geq\theta>0.
\]
Consequently,
\[
\Psi_a(w_n)\geq\frac{\theta}{6}s_n^2\longrightarrow+\infty,
\]
by \eqref{eq 3.12}. This estimate uses $w_n$, not the weak limit whose cubic integral vanishes.
\end{proof}

\begin{remark}\label{H}
It is an immediate consequence of Proposition \ref{Prop 3.2} that $I_a$ verifies hypothesis $(H)$ in Theorem \ref{level}, with $\Gamma\equiv +\infty$. The nonempty boundary excludes the full-sphere setting; here the boundary barrier is infinite.
\end{remark}

\begin{proposition}              \label{Prop 3.4}
Suppose that $a \in \left(0, (\lambda_2 - \lambda_1)/\vert \Omega\vert\right)$, $\lambda < \lambda_1(a)$ and $g$ satisfies $(g_1)-(g_2)$. Then, the functional $I_a$ satisfies the $(PS)_{c_{\mathcal{N}_a}}$ condition, where $c_{\mathcal{N}_{a}} := \inf_{\mathcal{N}_a}I_a$.
\end{proposition}
\begin{proof} Let $\left\{u_n\right\}$ be a $(PS)_{c_{\mathcal{N}_{a}}}$ sequence of $I_a$. The coercivity on $\mathcal N_a$ alone does not concern an arbitrary PS sequence. For every $v\in H_0^1(\Omega)$, however,
\[
I_a(v)-\frac13I_a'(v)v
=\frac16\|v\|_a^2+\int_\Omega\left[\frac13g(v)v-G(v)\right]dx
\geq\frac{\theta}{6}\|v\|_a^2.
\]
Thus $\theta\|u_n\|_a^2/6\leq C+o(1)\|u_n\|_a$, which bounds $\{u_n\}$ in $H_0^1(\Omega)$. After extraction, $u_n\rightharpoonup u$ in $H_0^1$ and $u_n\to u$ in $L^2\cap L^3$. The growth bound on $g$ and these compact convergences give $g(u_n)\to g(u)$ in $L^2$ and $u_n^2\to u^2$ in $L^{3/2}$. Testing the derivative against $u_n-u$ yields
\[
\|u_n-u\|_a^2
=I_a'(u_n)(u_n-u)-\int_\Omega u_n^2(u_n-u)dx
+\int_\Omega g(u_n)(u_n-u)dx+o(1)\longrightarrow0,
\]
where the $o(1)$ is the weak pairing with $u$. This proves strong convergence.
\end{proof}

The main result of this subsection is as follows.
\begin{theorem}   \label{dito}                             \label{Teo 3.1}
Suppose that $a \in \left(0, (\lambda_2 - \lambda_1)/\vert \Omega\vert\right)$, $\lambda < \lambda_1(a)$ and $g$ satisfies $(g_1)-(g_2)$. Then, problem \eqref{Prob P_2} admits a ground-state solution whose negative part is nontrivial: it changes sign or is nonpositive and nonzero.
\end{theorem}
\begin{proof} 
It follows from Lemma \ref{Lem 3.2}, Remark \ref{H}, Propositions \ref{Prop 3.1}, \ref{Prop 3.4} and Theorem \ref{level} that problem \eqref{Prob P_2} has a nontrivial ground state solution. Furthermore, since $u\in\mathcal{N}_a\subset \mathcal{A}$ (see Lemma \ref{Lem 1.2}) we conclude that 
\[\int_{\Omega}u^3 \ dx < 0,\]
and then $u$ has a nontrivial negative part. Thus it is nonpositive and nonzero or changes sign.       
                        
\end{proof}
\subsection{An integro-differential problem: Superquadratic growth in a subdomain}\label{32}

In this section, we are interested in studying the existence of a ground-state solution for the following problem
\begin{equation}	     \label{Prob P_3}\tag{$P_2$}
\left\{
\begin{array}{rcl}
-\Delta u + a \displaystyle\int_{\Omega} u \ dx &=& \lambda u + (1 - \chi_{\Omega_0}(x))g(u) + \chi_{\Omega_0}(x)f(u) \ \ \mbox{in} \ \Omega,\\[0.3cm]
u& =& 0 \ \  \mbox{on} \ \partial \Omega,
\end{array}
\right.
\end{equation}
where $\Omega \subset \R^N$ is a bounded domain with a smooth boundary, $a > 0$, $\lambda < \lambda_1(a)$, $\emptyset\neq\Omega_0 \subset\subset \Omega$ is open, and $f, g \in C(\R)$ with
$$
f(s)\leq c(1+s^{p-1}),
$$
for all $s>0$ and some constants $c>0$ and $2<p<2^{\ast}$, which satisfy:
\begin{itemize}
    \item[$(g_1)$] $g = 0 \ \mbox{in} \ (-\infty, 0], \  \displaystyle\lim_{s \rightarrow 0^+} g(s)/s = 0 \ \mbox{and} \ \lim_{s \rightarrow +\infty} g(s)/s = g_{\infty}\in[0,\kappa)$;

    \item[$(g_2)$]  $s \mapsto g(s)/s$ is a nondecreasing function in $(0, +\infty)$;

    \item[$(f_1)$] $f=0 \ \mbox{in} \ (-\infty, 0], \ \lim_{s \rightarrow 0^+} f(s)/s = f_0\in[0,\kappa) \ \mbox{and} \ \lim_{s \rightarrow +\infty} F(s)/s^2 = +\infty$, where $F(s):=\int^s_{0} f(\tau) \ d\tau$

    \item[$(f_2)$] $s \mapsto f(s)/ s$ is strictly increasing in $(0, +\infty)$;
\end{itemize}

Associated with problem \eqref{Prob P_3}, we have the energy functional $J_a: H^{1}_0(\Omega) \rightarrow \R$, given by 
\begin{equation}                     \label{eq 4.1}
J_a(u) = \frac{1}{2}\int_{\Omega}\left|\nabla u\right|^2 \ dx + \frac{a}{2} \left(\int_{\Omega} u \ dx\right)^2 - \frac{\lambda}{2}\int_{\Omega} u^2 \ dx - \int_{\Omega\backslash\Omega_{0}} G(u) \ dx - \int_{\Omega_0} F(u) \ dx,
\end{equation}
where $G(s) = \int^s_{0} g(\tau) \ d\tau$. It can be shown that $J_a \in C^1(H^{1}_{0}(\Omega), \R)$ with 
\[
\begin{aligned}
J_a'(u)v={}&\int_\Omega\nabla u\nabla v\,dx+a\int_\Omega u\,dx\int_\Omega v\,dx-\lambda\int_\Omega uv\,dx\\
&-\int_{\Omega\backslash\Omega_0}g(u)v\,dx-\int_{\Omega_0}f(u)v\,dx,
\qquad v\in H_0^1(\Omega).
\end{aligned}
\]

The Nehari set associated with $J_a$ is 
\[\mathcal{N}_a = \left\{u \in H^{1}_{0}(\Omega)\backslash\left\{0\right\}: \Vert u\Vert_a^2 = \int_{\Omega\backslash\Omega_{0}}g(u)u \ dx + \int_{\Omega_0}f(u)u \ dx\right\},\]
Under the same notation of section \ref{abstract setting}, by defining the set $\mathcal{A}$ as
\begin{equation}                      \label{eq 4.2}
\mathcal{A} = \left\{u \in H^{1}_{0}(\Omega): \vert supp (u^+) \cap\Omega_0\vert > 0\right\},
\end{equation}
we have
\[\mathcal{S}_\mathcal{A} = \left\{u \in \mathcal{S}: \vert supp (u^+) \cap\Omega_0\vert > 0\right\},\]
where $|A|$ denotes Lebesgue measure. In the formulas of this subsection, $supp(u^+)$ denotes the measurable set $[u>0]$, up to null sets, rather than its topological closure. Thus \eqref{eq 4.2} is equivalently $\mathcal A=\{u:\int_{\Omega_0}(u^+)^2dx>0\}$. 

\begin{lemma}                             \label{Lem 4.1}
Suppose that $a \in \left(0, (\lambda_2 - \lambda_1)/\vert \Omega\vert\right)$ and $\lambda < \lambda_1(a)$. Then, $\mathcal{A}$ is a nonempty open set and $s\mathcal{A}\subset \mathcal{A}$, for all $s > 0$.
\end{lemma}
\begin{proof}
The map $h(u)=\int_{\Omega_0}(u^+)^2dx$ is continuous, since the positive-part map is Lipschitz in $L^2$. It satisfies $h(tu)=t^2h(u)$ for $t>0$ and is positive on a nonzero nonnegative test function supported in $\Omega_0$. Thus $\mathcal A=h^{-1}((0,\infty))$ is a nonempty open cone.
\end{proof} 
\begin{proposition}                       \label{Prop 4.1}
Let $a \in (0, (\lambda_2 - \lambda_1)/\vert\Omega\vert)$, $\lambda < \lambda_1(a)$ and suppose that $(g_1)$-$(g_2)$, $(f_1)$-$(f_2)$ occur. Then, $(A_1)-(A_3)$ hold true.
\end{proposition}                                           
\begin{proof}
Assumption $(A_1)$ follows with the same $\varphi$, $\psi$ and $J$ as in Proposition \ref{Prop 3.1}. For $w\in\mathcal A$, the normalized derivative of
\[
\alpha_w(s)=\frac{s^2}{2}\|w\|_a^2-\int_{\Omega\setminus\Omega_0}G(sw)dx-\int_{\Omega_0}F(sw)dx
\]
is
\[
\frac{\alpha_w'(s)}s=\|w\|_a^2
-\int_{(\Omega\setminus\Omega_0)\cap[w>0]}\frac{g(sw)}{sw}w^2dx
-\int_{\Omega_0\cap[w>0]}\frac{f(sw)}{sw}w^2dx.
\]
It is strictly decreasing by $(g_2)$--$(f_2)$ and $|\Omega_0\cap[w>0]|>0$. Its limit at zero, justified by the growth bounds and dominated convergence, is
\begin{equation}\label{eq 4.5}
\|w\|_a^2-f_0\int_{\Omega_0}(w^+)^2dx
\geq(1-f_0/\kappa)\|w\|_a^2>0.
\end{equation}
Since $f(t)/t\geq2F(t)/t^2\to\infty$ as $t\to\infty$, Fatou's lemma on $\Omega_0\cap[w>0]$ gives
\begin{equation}\label{eq 4.6}
\lim_{s\to\infty}\frac{\alpha_w'(s)}s=-\infty.
\end{equation}
Thus its unique zero has the signs in $(A_2)$. For $w\in\mathcal A^c\setminus\{0\}$, the $f$ integral vanishes and $\alpha_w'(s)/s\geq(1-g_\infty/\kappa)\|w\|_a^2>0$, which proves the second part of $(A_2)$.

For $(A_3)$, choose $\varepsilon>0$ with $f_0+\varepsilon<\kappa$. The growth assumptions give $f(t)\leq(f_0+\varepsilon)t+C_\varepsilon t^{p-1}$ for $t\geq0$ and $0\leq g(t)\leq g_\infty t$ there. Consequently,
\[
J_a'(u)u\geq\left[1-\frac{\max\{g_\infty,f_0+\varepsilon\}}\kappa\right]\|u\|_a^2
-C_\varepsilon'\|u\|_a^p.
\]
The positive quadratic coefficient and $p>2$ bound Nehari norms, hence $s_w$ on $\mathcal S_{\mathcal A}$, away from zero. If $w_n$ lies in a compact $\mathcal K\subset\mathcal S_{\mathcal A}$ and $s_n=s_{w_n}\to\infty$, extract $w_n\to w\in\mathcal K$ strongly and almost everywhere. The Nehari identity gives
\begin{equation}\label{eq 4.7}
1\geq\int_{\Omega_0\cap[w>0]}\frac{f(s_nw_n)}{s_nw_n}
\chi_{[w_n>0]}w_n^2dx.
\end{equation}
On this positive-measure set the integrand tends to $+\infty$, because $f(t)/t\to\infty$. Fatou's lemma contradicts \eqref{eq 4.7}, proving the compact-set upper bound.
\end{proof}

\begin{proposition}                  \label{Prop 4.3}
Let $a \in (0, (\lambda_2 - \lambda_1)/\vert\Omega\vert)$, $\lambda < \lambda_1(a)$ and suppose that $(g_1)$-$(g_2)$, $(f_1)$-$(f_2)$ occur. Then, $J_a$ is bounded from below and coercive in $\mathcal{N}_a$.
\end{proposition}
\begin{proof}
Let $u \in \mathcal{N}_a$, then
\[
\Vert u\Vert^{2}_{a} = \int_{\Omega\backslash\Omega_{0}}g(u)u \ dx + \int_{\Omega_{0}}f(u)u \ dx.
\]
Now, hypotheses $(g_1)-(g_2)$ imply
\[J_a(u) = \int_{\Omega\backslash\Omega_0}\left[\frac{1}{2}g(u)u - G(u)\right] \ dx + \int_{\Omega_0}\left[\frac{1}{2}f(u)u - F(u)\right] \ dx \geq \int_{\Omega_0}\left[\frac{1}{2}f(u^+)u^+ - F(u^+)\right] \ dx. \]
From Lemma \ref{Lem 1.2} and Proposition \ref{Prop 4.1}, we have $\mathcal{N}_a \subset\mathcal{A}$ and, consequently, by $(A_2)$, $J_a(u) > 0$,  for all  $u \in \mathcal{N}_a$.

\medskip

To prove that $J_a$ is coercive on $\mathcal{N}_a$, let $\left\{u_n\right\}\subset \mathcal{N}_a$ be such that $\Vert u_n\Vert_a \rightarrow +\infty$, as $n \rightarrow +\infty$. Define $w_n := u_n/\Vert u_n\Vert_a$. Then 
$\left\{w_n\right\} \subset \mathcal{S}_{\mathcal{A}}$ and $s_{w_n} = \Vert u_n\Vert_a \rightarrow +\infty$. Let $w_n \rightharpoonup w$ in $H^{1}_{0}(\Omega)$. We are going to prove that $w^+ = 0$ in $\Omega_0$. In fact, suppose by contradiction that $w^+ \neq 0$ in $\Omega_0$. Then,
\[
0 \leq \frac{J_a(u_n)}{\Vert u_n\Vert^2_a} \leq \frac{1}{2} - \int_{\Omega_0\cap [w_n>0]}\frac{F(s_{w_n}w_n)}{(s_{w_n}w_n)^2}w_n^2 \ dx,
\]
After a subsequence, $w_n\to w$ almost everywhere. On $E=\Omega_0\cap[w>0]$, of positive measure, eventually $w_n>0$, $s_{w_n}w_n\to+\infty$, and $w_n^2\to w^2>0$. The nonnegative integrand in the last expression therefore tends to infinity on $E$. Fatou's lemma makes its integral diverge, contradicting the nonnegative left side. Therefore, $w^+ = 0$ in $\Omega_0$. Thus,
\[
J_a(u_n) = J_a(s_{w_n}w_n) \geq J_a(s w_n) = \frac{s^2}{2} - \int_{\Omega\backslash\Omega_0}G(sw_n) \ dx - \int_{\Omega_0}F(sw_n) \ dx.
\]
For each fixed $s>0$, compact convergence in $L^2$ and $L^p$ and the growth bounds yield convergence of the integrals of $G(sw_n)$ and $F(sw_n)$. Since $w^+=0$ on $\Omega_0$, the latter limit is zero. Consequently, it follows from $(g_1)$ that
\[
\liminf_{n \rightarrow +\infty} J_a(u_n) \geq \frac{s^2}{2} - \int_{\Omega\backslash\Omega_{0}}G(sw^+) \ dx \geq \frac{1}{2}\left(1 - \frac{g_{\infty}}{\kappa}\right)s^2,
\]
for all $s>0$.  Therefore $J_a(u_n) \rightarrow +\infty$, since  $0\leq g_{\infty} < \kappa$. 
                
\end{proof}

\begin{proposition}                \label{Porp 4.2}
Let $a \in (0, (\lambda_2 - \lambda_1)/\vert\Omega\vert)$, $\lambda < \lambda_1(a)$ and suppose that $(g_1)$-$(g_2)$, $(f_1)$-$(f_2)$ occur. If $\left\{w_n\right\} \subset \mathcal{S}_\mathcal{A}$ is a sequence satisfying $\mathrm{dist}(w_n, \partial\mathcal{S}_\mathcal{A}) \rightarrow 0$, then $s_{w_n} \rightarrow +\infty$ and $\Psi(w_n) \rightarrow +\infty$, as $n \rightarrow +\infty$.
\end{proposition}
\begin{proof}
Write $d_n=\mathrm{dist}(w_n,\partial\mathcal S_{\mathcal A})$. Since $v^+=0$ on $\Omega_0$ for each boundary direction $v$, the Sobolev embedding gives
\begin{equation}\label{eq 4.8}
\int_{\Omega_0}(w_n^+)^r dx\leq C_r d_n^r,
\qquad 1\leq r<2^*,\quad r\text{ finite}.
\end{equation}
The growth bounds $0\leq F(t)\leq C(t^2+t^p)$ and $0\leq G(t)\leq g_\infty t^2/2$, $t\geq0$, consequently give, for each fixed $s>0$,
\[
\Psi(w_n)\geq J_a(sw_n)\geq(1-g_\infty/\kappa)\frac{s^2}{2}
-C(s^2d_n^2+s^pd_n^p).
\]
Passing to the lower limit as $n\to\infty$, we obtain 
\[
\liminf_{n \rightarrow +\infty}\Psi(w_n) \geq \left(1 - \frac{g_\infty}{\kappa}\right)\frac{s^2}{2}, \ \ \ \forall \ s > 0.
\]
Proving that $\Psi(w_n) \rightarrow +\infty$. Once $\Psi(w_n)\leq \frac{1}{2}s_{w_n}^2$, for all $n \in \mathbb{N}$, we conclude that $s_{w_n} \rightarrow +\infty$.
        
\end{proof}

\begin{remark}\label{disso}
As in the previous subsection, it is a consequence of Proposition \ref{Porp 4.2} that $J_a$ satisfies assumption $(H)$ in Theorem \ref{level}, with $\Gamma\equiv\infty$. Again, a proper spherical cone is involved, with an infinite boundary barrier.
\end{remark}

\medskip

As in the previous section we will prove compactness for ambient Palais--Smale sequences contained in $\mathcal N_a$ at the least energy level
\[
d_{\mathcal{N}_a} := \inf_{u \in \mathcal{S}_{\mathcal{A}}} \Psi_a(u)= \inf_{u \in \mathcal{N}_a}J_a(u).
\]
\begin{proposition}            \label{Prop 4.4}
Let $a \in (0, (\lambda_2 - \lambda_1)/\vert\Omega\vert)$, $\lambda < \lambda_1(a)$ and suppose $(g_1)-(g_2)$ and $(f_1)-(f_2)$ occur. Then, the functional $J_a$ satisfies the $(PS)^{\mathcal N_a}_{d_{\mathcal N_a}}$ condition.

\end{proposition}
\begin{proof}
Let $u_n\in\mathcal N_a$ be an ambient PS sequence at $d_{\mathcal N_a}$. Proposition~\ref{Prop 4.3} bounds it in $H_0^1$. Extract $u_n\rightharpoonup u$ there and $u_n\to u$ in $L^2\cap L^p$. The growth assumptions give $g(u_n)\to g(u)$ in $L^2$ and $f(u_n)\to f(u)$ in $L^{p/(p-1)}$. Thus
\[
\|u_n-u\|_a^2=J_a'(u_n)(u_n-u)
+\int_{\Omega\setminus\Omega_0}g(u_n)(u_n-u)dx
+\int_{\Omega_0}f(u_n)(u_n-u)dx+o(1)\longrightarrow0.
\]
The weak pairing with $u$ gives the $o(1)$ term, proving strong convergence.
\end{proof}                 
The main result of this subsection is the following.
\begin{theorem}             \label{Teo 4.1}
   Let $a \in (0, (\lambda_2 - \lambda_1)/\vert\Omega\vert)$, $\lambda < \lambda_1(a)$ and suppose $(g_1)-(g_2)$ and $(f_1)-(f_2)$ occur. Then, problem \eqref{Prob P_3} admits a non-trivial ground-state solution, whose positive part is nontrivial in $\Omega_0$.
\end{theorem}
\begin{proof}
In view of Remark \ref{disso}, Propositions \ref{Prop 4.1}, \ref{Prop 4.3}, \ref{Prop 4.4} and Theorem \ref{level}, the result follows.
\end{proof}

\medskip

In the last two applications, we chose $\Gamma\equiv+\infty$ in order to verify hypothesis $(H)$. In the next subsection, we are going to provide a problem for which hypothesis $(H)$ is satisfied with $\Gamma\not\equiv+\infty$. The resulting criterion is related to the asymptotic problems in \cite{LZ} and to the Nehari approach in \cite{NPS}; the contribution considered below is the quantified boundary-energy barrier for $p>1$. We do not assert a strict implication between \eqref{eq b} and the full hypotheses of \cite{LZ}.

\subsection{ A \texorpdfstring{$p$}{p}-Laplacian problem with asymptotic behaviour at infinity}\label{asymptotic-p}

Consider the $p$-Laplacian problem
\begin{equation}		                       \label{Prob P_1}\tag{$P_3$}
\left\{
\begin{array}{rccccl}
-\Delta_p u &=&f(u), \ \mbox{in} \ \Omega,\\
u &=&0, \ \mbox{on} \ \partial \Omega,
\end{array}
\right.
\end{equation}
where $\Omega \subset \R^N$ is a bounded domain with a smooth boundary, $1 < p < +\infty$ and $f \in C(\R)$ is an odd function. Set $q(s):=f(s)/(|s|^{p-2}s)$ for $s\neq0$, with $q(0)=\alpha$. This is an even quotient. We assume:
\begin{itemize}
\item[$(f_1)$] $\lim_{s\to0}q(s)=\alpha\in[0,\lambda_1)$ and $\lim_{|s|\to\infty}q(s)=\eta>\lambda_1$, where $\lambda_1$ is the first eigenvalue of $(-\Delta_p,W_0^{1,p}(\Omega))$;

\item[$(f_2)$] The function $q$ is strictly increasing for $s>0$;

\end{itemize}

The assumptions give $\alpha<q(s)<\eta$ for $s\neq0$ and $|f(s)|\leq\eta|s|^{p-1}$, which will be used in the convergence arguments.

It is well known that $\lambda_1$ can be represented as
\[
\lambda_1 = \inf\left\{\int_{\Omega}\vert\nabla u\vert^p \ dx: u \in W^{1,p}_{0}(\Omega), \int_{\Omega}\vert u\vert^p \ dx = 1\right\}.
\]
The norm in $W^{1,p}_{0}(\Omega)$ is defined by
\[\left\|u\right\| = \left(\int_{\Omega}\left|\nabla u\right|^{p} \ dx\right)^{1/p}.\]

The energy functional $I: W^{1,p}_{0}(\Omega) \rightarrow \R$ associated to problem \eqref{Prob P_1} is given by
\begin{equation}                       \label{eq 2.1}
I(u) = \frac{1}{p}\left\|u\right\|^p - \int_{\Omega}F(u) \ dx,
\end{equation}
where $F(s) = \int_{0}^{s}f(\tau) \ d\tau$. It can be shown that $I \in C^{1}(W^{1,p}_{0}(\Omega), \R)$ with
\begin{equation}                       \label{eq 2.2}
I'(u)v = \int_{\Omega}\left|\nabla u\right|^{p-2}\nabla u\nabla v \ dx - \int_{\Omega}f(u) v \ dx, \ \ \ \forall \ v \in W^{1,p}_{0}(\Omega).
\end{equation}
The Nehari set associated with the functional $I$ is the following set:
\[\mathcal{N} = \left\{u \in W^{1,p}_{0}(\Omega)\backslash\left\{0\right\}: \left\|u\right\|^p = \int_{\Omega}f(u)u \ dx\right\}.\]

Furthermore, we shall associate with problem \eqref{Prob P_1} the set $\mathcal{A}$ defined by
\begin{equation}                                        \label{eq 2.3}
\mathcal{A} = \left\{u \in W^{1,p}_{0}(\Omega): \left\|u\right\|^p < \eta\int_{\Omega} \left|u\right|^p \ dx\right\}.
\end{equation}                                           

The result below provides some information about the set $\mathcal{A}$.

\begin{lemma}                       \label{Lem 2.1}
Suppose that $f$ satisfies $(f_1)-(f_2)$. Then, $\mathcal{A}$ is a nonempty open set and $s\mathcal{A}\subset\mathcal{A}$, for all $s > 0$.
\end{lemma} 
\begin{proof}
The continuous $p$-homogeneous map $h(u)=\|u\|^p-\eta\|u\|_p^p$ gives $\mathcal A=h^{-1}(( -\infty,0))$, an open cone. Since $\lambda_1<\eta$, a first eigenfunction belongs to $\mathcal A$.
\end{proof}

In what follows, by using the same notation of section \ref{abstract setting}, we have 
\[\mathcal{S}_\mathcal{A} = \left\{u \in \mathcal{S}: 1 < \eta\int_{\Omega} \left|u\right|^p \ dx\right\}.\]

\begin{proposition}                         \label{Prop 2.1}
Suppose that $f$ satisfies $(f_1)-(f_2)$. Then, assumptions $(A_1)-(A_3)$ are satisfied.
\end{proposition}
\begin{proof}
For $(A_1)$ take $\varphi(t)=t^{p-1}$, $\psi(u)=\|u\|^p/p$ and $J(u)v=\int_\Omega|\nabla u|^{p-2}\nabla u\cdot\nabla v\,dx$. The bounded, even quotient $q$, defined above, gives
\[
\frac{\alpha_w'(s)}{s^{p-1}}=\|w\|^p-\int_\Omega q(sw)|w|^pdx.
\]
For $w\in\mathcal A$ this is strictly decreasing, and dominated convergence gives the limits $\|w\|^p-\alpha\|w\|_p^p>0$ at zero and $\|w\|^p-\eta\|w\|_p^p<0$ at infinity. Its unique zero has the strict signs in $(A_2)$ and is the global maximum of $\alpha_w$. Outside the cone, $q(sw)<\eta$ on $[w\neq0]$ makes the same derivative strictly positive.

For the lower bound in $(A_3)$, suppose $w_n\in\mathcal S_{\mathcal A}$ and $s_n=s_{w_n}\to0$. Extract $w_n\rightharpoonup w$ in $X$ and $w_n\to w$ in $L^p$ and almost everywhere. Since $q$ is bounded and $q(0)=\alpha$, the Nehari identity implies
\[
1=\int_\Omega q(s_nw_n)|w_n|^pdx\longrightarrow\alpha\|w\|_p^p\leq\alpha/\lambda_1<1,
\]
a contradiction. Indeed, $|w_n|^p\to|w|^p$ in $L^1$, so one may replace that factor by $|w|^p$ and apply dominated convergence. For the upper bound, if $w_n$ belongs to a compact $\mathcal K\subset\mathcal S_{\mathcal A}$ and $s_n\to\infty$, extract $w_n\to w\in\mathcal K$ strongly in $X$. The same argument gives $1=\eta\|w\|_p^p$: on $[w\neq0]$, $q(s_nw_n)\to\eta$, and on $[w=0]$ the limiting $L^1$ factor vanishes. This contradicts $w\in\mathcal S_{\mathcal A}$.
\end{proof}

\begin{remark}
Unlike the two previous applications, in the present problem we are unable to prove that the functional is coercive on the Nehari set, because of the asymptotic behavior of the nonlinearity involved. When $\beta$ defined below is finite, the barrier $\Gamma$ is finite on the nonempty boundary; hypothesis $(H)$ allows this case without requiring boundary divergence. This fact brings considerable technical difficulties for the definition of the function $\Gamma$, making the problem particularly delicate. Therefore, in what follows we shall devote the next lemmas to overcoming this difficulty, in order to define a suitable function $\Gamma$ and to prove that hypothesis $(H)$ holds.
\end{remark}

Fix a finite exponent $r>p$ with $r\leq p^*=Np/(N-p)$ if $p<N$; when $p\geq N$, choose any finite $r>p$. The following Sobolev estimates use this fixed $r$, so they also cover $p=N$ without assuming an embedding into $L^\infty$.

\begin{lemma}                 \label{Lem 2.2} 
Suppose that $f$ satisfies $(f_1)-(f_2)$. Then, 
\[\inf_{w \in \partial\mathcal{S}_{\mathcal{A}}} \vert [w \neq 0]\vert\geq \left(\frac{S_r^p}{\eta}\right)^{\frac{r}{r - p}},\]
where $\eta$ was introduced in $(f_1)$ and 
\[S_r = \inf_{w \ \in W^{1,p}_{0}(\Omega)\backslash\left\{0\right\}} \frac{(\int_{\Omega} \vert \nabla w\vert^{p}\ dx)^{1/p}}{(\int_{\Omega} \vert w\vert^{r}\ dx)^{1/r}}\cdot\]
\end{lemma}
\begin{proof}
For $w\in\partial\mathcal S_{\mathcal A}$, H\"older's inequality and the Sobolev embedding give
\[
1=\eta\|w\|_p^p\leq\eta|[w\neq0]|^{(r-p)/r}\|w\|_r^p
\leq\frac{\eta}{S_r^p}|[w\neq0]|^{(r-p)/r}.
\]
Rearranging proves the estimate.
\end{proof}


Before proving our next result, we are going to define the following ``number''
\begin{equation}                                        \label{eq 2.6}
\beta := \lim_{\vert s\vert \rightarrow +\infty}\left[\frac{1}{p}f(s)s - F(s)\right].
\end{equation}
To justify this assertion without differentiating $f$, write $K(t)=f(t)t/p-F(t)$ for $t>0$. Integration by parts for the continuous increasing quotient $q$ gives
\[
K(t)=\frac1p\int_{(0,t]}s^p\,dq(s).
\]
Hence $K$ is even, nonnegative and strictly increasing in $|t|$. In particular, $\beta$ is well defined and belongs to $(0,+\infty]$. In what follows, we say that the hypothesis $(\beta)$ is satisfied if
\begin{equation}           \label{eq b} \tag{$\beta$}
    \beta > \frac{1}{p}(\eta - \alpha)\left(\frac{\tau_{\mathcal{A}}}{S_r}\right)^p\left(\frac{\eta}{S_r^p}\right)^{\frac{r}{r - p}}\vert\Omega\vert^{\frac{r - p}{r}},
\end{equation}
where $\tau_{\mathcal{A}} = \inf_{w \in\mathcal{S}_\mathcal{A}}s_w > 0$. Observe that if $\beta=+\infty$, then hypothesis \eqref{eq b} is trivially satisfied. Examples of functions $f$ satisfying assumption \eqref{eq b} with $\beta<+\infty$ were provided in \cite{NPS} for the case $p=2$. Example~\ref{p-explicit-primitive} below gives an explicit primitive for every $p>1$ and verifies all the assumptions, including \eqref{eq b}, with $\beta<+\infty$.

\medskip

\begin{lemma}                         \label{Lem 2.4}
Suppose that $f$ satisfies $(f_1)-(f_2)$ and \eqref{eq b} occurs. Then,    
    \[
    c_\mathcal{N} < \beta\inf_{w \in \partial\mathcal{S}_{\mathcal{A}}}\vert[w \neq 0]\vert,
    \]
    where $c_{\mathcal{N}}:=\inf_{u\in\mathcal{N}}I(u)$.
\end{lemma}
\begin{proof}
   Let $w \in \mathcal{S}_{\mathcal{A}}$. It follows from $(f_1)-(f_2)$ that
   
    \[
c_{\mathcal{N}} \leq \int_{[w \neq 0]}\left[\frac{1}{p}q(s_ww) - \frac{F(s_w w)}{\vert s_w w\vert^{p}}\right]\vert s_w w\vert^{p} \ dx< \frac{1}{p}(\eta - \alpha)s^p_w\int_{\Omega}\vert w\vert^{p} \ dx. 
    \]
For $\|w\|=1$, H\"older and Sobolev give $\|w\|_p^p\leq|\Omega|^{(r-p)/r}\|w\|_r^p\leq|\Omega|^{(r-p)/r}/S_r^p$. Hence
\[
c_{\mathcal N}\leq\frac{\eta-\alpha}{p}\frac{s_w^p}{S_r^p}|\Omega|^{(r-p)/r},
\qquad w\in\mathcal S_{\mathcal A}.
\]
Thus, 
\begin{equation}                      \label{eq 2.11}
    c_\mathcal{N} \leq \frac{1}{p}(\eta - \alpha)\frac{\tau^p_{\mathcal{A}}}{S_r^p}\vert\Omega\vert^{\frac{r - p}{r}},
\end{equation}
where $\tau_{\mathcal{A}} := \inf_{w \in \mathcal{S}_{\mathcal{A}}}s_w > 0$. Now, Lemma \ref{Lem 2.2}, \eqref{eq b} and \eqref{eq 2.11} imply
$$                                      
    \beta\inf_{w \in \partial\mathcal{S}_{\mathcal{A}}}\vert[w \neq 0]\vert\geq \beta \left(\frac{S_r^p}{\eta}\right)^{\frac{r}{r - p}} > \frac{1}{p}(\eta - \alpha)\left(\frac{\tau_{\mathcal{A}}}{S_r}\right)^p\vert \Omega\vert^{\frac{r - p}{r}} \geq c_{\mathcal{N}}.
$$                
\end{proof}

\begin{lemma}                                        \label{Prop 2.2}
Suppose that $f$ satisfies $(f_1)$-$(f_2)$ and (\ref{eq b}) occurs. If $\left\{w_n\right\} \subset \mathcal{S}_\mathcal{A}$ is such that $\mathrm{dist}(w_n, \partial\mathcal{S}_\mathcal{A}) \ \rightarrow \ 0$, then, up to a subsequence, there exists $w \in W^{1,p}_{0}(\Omega)\backslash\left\{0\right\}$ such that          
\[w_n \rightharpoonup w \ \mbox{in} \ W^{1,p}_{0}(\Omega),\]
\[s_{w_n} \rightarrow +\infty, \ \ \mbox{as} \ n \rightarrow +\infty\] 
and
\[\liminf_{n \rightarrow +\infty}\Psi(w_n) \geq \beta M_r,\qquad M_r:=\left(\frac{S_r^p}{\eta}\right)^{r/(r-p)}.\]
\end{lemma}
\begin{proof} Let $\left\{w_n\right\} \subset \mathcal{S}_\mathcal{A}$ be a sequence such that $\mathrm{dist}(w_n, \partial\mathcal{S}_\mathcal{A}) \rightarrow 0$. Since $\left\{w_n\right\}$ is bounded there exists $w \in W^{1,p}_{0}(\Omega)$ such that, up to a subsequence,
\begin{equation}                                    \label{eq 2.7}
w_n \rightharpoonup w \ \ \mbox{in} \ W^{1,p}_{0}(\Omega).
\end{equation}
Once $\mathrm{dist}(w_n, \partial\mathcal{S}_\mathcal{A}) \rightarrow 0$, there exists $z_n \in \partial\mathcal{S}_\mathcal{A}$ satisfying
\begin{equation}                                     \label{eq 2.8}
\left\|w_n - z_n\right\| \rightarrow 0, \ \ \ \mbox{as} \  n \rightarrow +\infty.
\end{equation}
Note that, 
\[\left|\eta\int_{\Omega}\vert w_n\vert^p \ dx - 1\right| = \eta\left|\int_{\Omega} (\vert w_n\vert^p  - \vert z_n\vert^p) \ dx\right| \leq  \eta\int_{\Omega}\vert \vert w_n\vert^p - \vert z_n\vert^p \vert \ dx.\]
For real $p>1$, use the pointwise inequality
\[
\bigl||a|^p-|b|^p\bigr|\leq p|a-b|(|a|+|b|)^{p-1}.
\]
Applying it to $w_n,z_n$ gives
\[
\left|\eta\int_\Omega|w_n|^pdx-1\right|
\leq p\eta\int_\Omega|w_n-z_n|(|w_n|+|z_n|)^{p-1}dx.
\]
H\"older's inequality with exponents $p$ and $p/(p-1)$ yields
\[
\left|\eta\int_\Omega|w_n|^pdx-1\right|
\leq p\eta |w_n-z_n|_p(|w_n|_p+|z_n|_p)^{p-1}.
\]
Since both vectors are on the unit sphere, Poincar\'e's inequality gives a constant $C>0$, independent of $n$, such that
\[
\left|\eta\int_\Omega|w_n|^pdx-1\right|\leq C\|w_n-z_n\|.
\]
From $(\ref{eq 2.8})$, it follows that 
\[\eta\int_{\Omega}\vert w_n\vert^p \ dx \rightarrow 1,\]
as $n \rightarrow +\infty$. Since $W_0^{1,p}(\Omega)$ embeds compactly into $L^p(\Omega)$ for every $p>1$ on a bounded domain, we have
\begin{equation}                        \label{eq 2.9}
\eta\int_{\Omega} \vert w\vert^p \ dx = 1,
\end{equation}
proving that $w \neq 0$.

Suppose by contradiction that, for some subsequence, $\left\{s_{w_n}\right\}$ is bounded. By Proposition $\ref{Prop 2.1}$, there exists $s_0 > 0$ such that, up to a subsequence, 
\begin{equation}                     \label{eq 2.10}
s_{w_n} \rightarrow s_0\ \ \mbox{as}\ \ n\to\infty.
\end{equation}      
It follows from (\ref{eq 2.7}), Sobolev compact embedding and Lebesgue dominated convergence theorem, that 
$$
s_0^{p-1} = \displaystyle\int_{\Omega}f(s_0 w)w \ dx.
$$
By (\ref{eq 2.10}), we have $w \neq 0$, and consequently it makes sense to write 
$$
1 = \displaystyle\int_{[w \neq 0]}q(s_0w) \vert w\vert^p \ dx.
$$
From hypothesis $(f_2)$, we obtain $1 < \eta\int_{\Omega}\vert w\vert^p \ dx$, which contradicts $(\ref{eq 2.9})$. Therefore, $s_{w_n} \rightarrow +\infty$. Finally, since $s_{w_n}w_n \in \mathcal{N}$, we have
\[\liminf_{n \rightarrow +\infty} \Psi(w_n) = \liminf_{n \rightarrow+\infty} \int_{\Omega} \left[\frac{1}{p}f(s_{w_n}w_n)s_{w_n}w_n - F(s_{w_n}w_n)\right] \ dx.\]
Again by $(f_2)$, we know that $s\mapsto (1/p)f(s)s-F(s)$ is nonnegative and increasing as a function of $|s|$. Thus,
\[\liminf_{n \rightarrow +\infty} \Psi(w_n) \geq \liminf_{n \rightarrow+\infty} \int_{[w \neq 0]} \left[\frac{1}{p}f(s_{w_n}w_n)s_{w_n}w_n - F(s_{w_n}w_n)\right] \ dx.\]
After a subsequence realizing the liminf, compact $L^p$ convergence gives $w_n(x)\to w(x)$ almost everywhere. On $[w\neq0]$, the fact that $s_{w_n}\to\infty$ implies $|s_{w_n}w_n(x)|\to\infty$, so $K(s_{w_n}w_n)\to\beta$. Fatou's lemma and $K\geq0$ give
\[
\liminf_n\Psi(w_n)\geq\beta|[w\neq0]|.
\]
Although $w$ need not have unit norm, weak lower semicontinuity gives $\|w\|\leq1$. Using \eqref{eq 2.9} and the same H\"older--Sobolev estimate as in Lemma \ref{Lem 2.2},
\[
1=\eta|w|_p^p\leq\eta |[w\neq0]|^{(r-p)/r}|w|_r^p
\leq\frac\eta{S_r^p}|[w\neq0]|^{(r-p)/r},
\]
we obtain $|[w\neq0]|\geq M_r$, proving the stated estimate, also when $\beta=+\infty$. Since $M_r$ is independent of the weak limit, applying this argument to a subsequence realizing the original liminf proves the bound for the full sequence.
                        
\end{proof}

Let $\Gamma:\partial\mathcal{S}_{\mathcal{A}}\to\R\cup\{+\infty\}$ be given by
\begin{equation}\label{gam}
\Gamma(u)=\beta\left|[u \neq 0]\right|,
\end{equation}
where $\left|[u \neq 0]\right|$ denotes the measure of $[u \neq 0]=\{x\in\Omega: u(x)\neq 0\}$.

\begin{proposition}\label{outra}
Suppose that $f$ satisfies $(f_1)-(f_2)$ and \eqref{eq b} occurs. Then, functional $I$ verifies hypothesis $(H)$, with $\Gamma$ defined as in \eqref{gam}.
\end{proposition}

\begin{proof}
For strong convergence $w_n\to w$ in $W_0^{1,p}$, an almost everywhere subsequence satisfies $\mathbf1_{[w\neq0]}\leq\liminf\mathbf1_{[w_n\neq0]}$. Applying Fatou along a subsequence realizing the liminf proves lower semicontinuity of $|[w\neq0]|$, hence of $\Gamma$ for finite $\beta$. For $\beta=+\infty$, define $\Gamma\equiv+\infty$ on the relative boundary.

Lemma \ref{Lem 2.4} gives $c_{\mathcal N}<\inf_{\partial\mathcal S_{\mathcal A}}\Gamma$. For the pointwise boundary inequality in $(H)$, suppose $w_n\in\mathcal S_{\mathcal A}$ converges strongly to $w\in\partial\mathcal S_{\mathcal A}$. Lemma \ref{Prop 2.2} gives $s_{w_n}\to\infty$. On $[w\neq0]$, almost everywhere convergence then gives $K(s_{w_n}w_n)\to\beta$. The same nonnegative Fatou argument yields
\[
\liminf_n\Psi(w_n)\geq\beta|[w\neq0]|=\Gamma(w),
\]
with the usual interpretation if $\beta=+\infty$. This proves $(H)$.  
\end{proof}

The following result ensures that $I$ satisfies the Palais-Smale condition in a given interval.
\begin{proposition}                        \label{Prop 2.3}
Suppose that $f$ satisfies $(f_1)-(f_2)$ and \eqref{eq b} occurs. Then, the functional $I$ satisfies the $(PS)^{\mathcal N}_c$ condition, for any $c \in [c_\mathcal{N}, \inf_{w \in \partial\mathcal{S}_\mathcal{A}}\Gamma(w))$.
\end{proposition}
\begin{proof} Let $\left\{u_n\right\} \subset \mathcal{N}$ be a  Palais-Smale sequence for $I$ at the level $c$ . The sequence $\left\{u_n\right\}$ is bounded. In fact, suppose by contradiction that, up to  a subsequence, 
$ \left\|u_n\right\| \rightarrow +\infty, \ \ \mbox{as} \ n \rightarrow +\infty$. Let $v_n := \frac{u_n}{\Vert u_n\Vert} = m^{-1}(u_n) \in \mathcal{S}_{\mathcal{A}}$. The sequence $\left\{v_n\right\}$ is bounded in $W^{1,p}_{0}(\Omega)$ and 
\begin{equation}                                       \label{eq 2.15}
    \Psi(v_n) \rightarrow c, \ \ \ \mbox{as} \ n \rightarrow +\infty.
\end{equation}
Consequently, there is $v \in W^{1,p}_{0}(\Omega)$ such that
\begin{equation}                                       \label{eq 2.16}
    v_n \rightharpoonup v \ \ \mbox{in} \ W^{1,p}_{0}(\Omega).
\end{equation}
Suppose $v = 0$. Since $\left\{\Psi(v_n)\right\}$ is bounded, it follows that there exists a constant $K > 0$ such that
\begin{equation}                                        \label{eq 2.17}
    K > \Psi(v_n) = I(s_{v_n}v_n) \geq I(sv_n) = \frac{s^p}{p} - \int_{\Omega}F(sv_n) \ dx, \ \ \ \forall \ s > 0.
\end{equation}
Now, by using $(\ref{eq 2.16})$ and passing to the limit as $n \rightarrow +\infty$ in $(\ref{eq 2.17})$, we obtain 
$$
K \geq s^p/p,
$$ 
for all $s > 0$, which is a contradiction. Thus, $ v \neq 0$. Let $t_n=\|u_n\|\to\infty$ and denote by $A:W_0^{1,p}(\Omega)\to W^{-1,p'}(\Omega)$, with $p'=p/(p-1)$, the $p$-Laplacian operator. Dividing the ambient PS equation by $t_n^{p-1}$ gives
\begin{equation}\label{eq 2.19}
A(v_n)-h_n\longrightarrow0\quad\mbox{in }W^{-1,p'},
\qquad h_n:=\frac{f(t_nv_n)}{t_n^{p-1}}.
\end{equation}
We first identify the limit of the source, before passing through the nonlinear operator. On $[v\neq0]$, almost everywhere convergence implies $|t_nv_n|\to\infty$, and therefore
\[
h_n=q(t_nv_n)|v_n|^{p-2}v_n\longrightarrow\eta|v|^{p-2}v.
\]
On $[v=0]$, the same limit is zero because $|h_n|\leq\eta|v_n|^{p-1}$. Moreover,
\begin{equation}\label{eq 2.20}
|h_n|^{p'}\leq\eta^{p'}|v_n|^p,\qquad p'=p/(p-1).
\end{equation}
Strong convergence $v_n\to v$ in $L^p$ makes the right side uniformly integrable. Vitali's theorem consequently yields
\begin{equation}\label{eq 2.21}
h_n\longrightarrow\eta|v|^{p-2}v\quad\mbox{in }L^{p'}(\Omega).
\end{equation}
Testing \eqref{eq 2.19} against $v_n-v$ gives $\langle A(v_n),v_n-v\rangle\to0$. The type $(S_+)$ property of the $p$-Laplacian yields strong convergence $v_n\to v$ in $W_0^{1,p}$. In particular, $\|v\|=1$. Only after this strong convergence do we pass to the limit in \eqref{eq 2.19}, obtaining
\begin{equation}\label{eq 2.22}
\int_\Omega|\nabla v|^{p-2}\nabla v\cdot\nabla w\,dx
=\eta\int_\Omega|v|^{p-2}vw\,dx
\quad\mbox{for every }w\in W_0^{1,p}(\Omega).
\end{equation}
Taking $w=v$, we find
\[
1=\|v\|^p=\eta|v|_p^p.
\]
Thus $v\notin\mathcal S_{\mathcal A}$; since $v_n\in\mathcal S_{\mathcal A}$ converges strongly to $v$, it belongs to the relative boundary. No enumeration of the complete $p$-Laplacian spectrum is used. Proposition \ref{outra} now gives
\[
c=\lim_n\Psi(v_n)\geq\Gamma(v)\geq\inf_{\partial\mathcal S_{\mathcal A}}\Gamma,
\]
contrary to the assumed interval for $c$. This proves that $\{u_n\}$ is bounded.

\medskip

Extract $u_n\rightharpoonup u$ in $W_0^{1,p}$ and $u_n\to u$ in $L^p$. The growth bound $|f(s)|\leq\eta|s|^{p-1}$ and continuity imply $f(u_n)\to f(u)$ in $L^{p'}$. Testing the original PS equation against $u_n-u$ therefore gives
\begin{equation}\label{eq 2.27}
\langle A(u_n),u_n-u\rangle
=\int_\Omega f(u_n)(u_n-u)dx+o(1)\longrightarrow0.
\end{equation}
A second application of $(S_+)$ yields $u_n\to u$ in $W_0^{1,p}$, completing the compactness proof.
\end{proof}

Our main result in this subsection is the following.
\begin{theorem}                         \label{Teo 2.1}
Suppose that $f$ satisfies $(f_1)$-$(f_2)$ and \eqref{eq b} occurs. Then, problem \eqref{Prob P_1} has a ground-state solution with defined sign.
\end{theorem}
\begin{proof} It is a consequence of Lemma \ref{Lem 2.1}, Propositions \ref{Prop 2.1}, \ref{outra}, \ref{Prop 2.3} and Theorem \ref{level}. Finally, to show that $u$ does not change sign, observe that write $u^+=\max\{u,0\}$ and $u^-=\max\{-u,0\}$. If both are nontrivial, testing the equation against $u^+$ and $-u^-$ gives $u^+,-u^-\in\mathcal N$. Since $f$ is odd, its primitive is even. Furthermore, 
\[c_\mathcal{N} = I(u) = I(u^+) + I(-u^-) \geq 2c_\mathcal{N},\]
which contradicts $c_{\mathcal N}>0$ from Theorem \ref{level}; thus we have $u^{+}=0$ or $u^{-}=0$.
\end{proof}

The following explicit primitive verifies all the assumptions of this subsection with a finite $\beta$. Its two scales make it possible to control $\tau_{\mathcal A}$ independently of the parameter used to enlarge the boundary barrier.
\begin{example}\label{p-explicit-primitive}
Fix $p>1$ and the finite Sobolev exponent $r$ chosen above. With the same $\lambda_1$ and $S_r$, define
\begin{equation}\label{p-example-parameters}
\begin{aligned}
T&=\left(\frac{4\lambda_1|\Omega|}{\mathrm e}\right)^{1/p},\\
D_r&=3\left(\frac{T}{S_r}\right)^p
\left(\frac{3\lambda_1}{S_r^p}\right)^{\frac r{r-p}}
|\Omega|^{\frac{r-p}{r}},
\qquad L=(1+D_r)^{1/p},
\end{aligned}
\end{equation}
where $\mathrm e=\exp(1)$. Then the even primitive
\begin{equation}\label{p-example-primitive}
F_L(s)=\frac{\lambda_1}{p}\left[3|s|^p-2(1-e^{-|s|^p})
-L^p(1-e^{-|s|^p/L^p})\right]
\end{equation}
has an odd continuous derivative $f_L$, satisfies $(f_1)$--$(f_2)$ and \eqref{eq b}, and has
\[
\alpha=0,\qquad \eta=3\lambda_1,\qquad
\beta=\frac{\lambda_1}{p}(2+L^p)<+\infty.
\]
Consequently Theorem \ref{Teo 2.1} applies, with a finite barrier $\Gamma(w)=\beta|[w\neq0]|$.
\end{example}
\begin{proof}
Differentiating \eqref{p-example-primitive} for $s\neq0$ gives
\begin{equation}\label{p-example-reaction}
f_L(s)=\lambda_1|s|^{p-2}s
\left[3-2e^{-|s|^p}-e^{-|s|^p/L^p}\right],\qquad f_L(0)=0.
\end{equation}
Indeed, the three differentiated terms inside the brackets of \eqref{p-example-primitive} are, respectively,
\[
3p|s|^{p-2}s,\qquad
-2p|s|^{p-2}s e^{-|s|^p},\qquad
-p|s|^{p-2}s e^{-|s|^p/L^p}.
\]
Also $F_L(0)=0$ and $|F_L(s)|=O(|s|^p)$ as $s\to0$, so $F_L'(0)=0$ since $p>1$. Formula \eqref{p-example-reaction} shows that $f_L$ is odd and continuous, including at zero. Thus $F_L(s)=\int_0^s f_L(t)\,dt$ on the whole real line.

For $s\neq0$, its even quotient is
\[
q_L(s)=\frac{f_L(s)}{|s|^{p-2}s}
=\lambda_1\left[3-2e^{-|s|^p}-e^{-|s|^p/L^p}\right].
\]
Hence $q_L(s)\to0$ as $s\to0$, $q_L(s)\to3\lambda_1$ as $|s|\to\infty$, and for $s>0$,
\[
q_L'(s)=p\lambda_1s^{p-1}
\left[2e^{-s^p}+L^{-p}e^{-s^p/L^p}\right]>0.
\]
This verifies $(f_1)$--$(f_2)$, with $0<q_L(s)<3\lambda_1$ for $s\neq0$ and $|f_L(s)|\leq3\lambda_1|s|^{p-1}$.

To compute the barrier coefficient, put $z=|s|^p$. Subtracting \eqref{p-example-primitive} from $f_L(s)s/p$ yields
\begin{equation}\label{p-example-defect}
\begin{aligned}
\frac1p f_L(s)s-F_L(s)
=\frac{\lambda_1}{p}\bigg[
&2\big(1-(1+z)e^{-z}\big)\\
&+L^p\big(1-(1+z/L^p)e^{-z/L^p}\big)\bigg].
\end{aligned}
\end{equation}
Since $1-(1+t)e^{-t}\geq0$ for $t\geq0$ and tends to $1$ as $t\to\infty$, the defect is nonnegative and its limit is $\beta=\lambda_1(2+L^p)/p$.

It remains to verify \eqref{eq b}, including its dependence on $\tau_{\mathcal A}$. Let $e_p$ be a first $p$-Laplacian eigenfunction normalized by $\|e_p\|=1$. Then
\[
\int_\Omega|e_p|^pdx=\frac1{\lambda_1},\qquad e_p\in\mathcal S_{\mathcal A},
\]
because $\eta=3\lambda_1$. For every $L>0$,
\[
q_L(s)\geq2\lambda_1(1-e^{-|s|^p}).
\]
The elementary bound $y e^{-T^p y}\leq1/(\mathrm e T^p)$, $y\geq0$, therefore gives
\begin{equation}\label{p-example-scale-bound}
\begin{aligned}
\int_\Omega q_L(Te_p)|e_p|^pdx
&\geq2\lambda_1\int_\Omega(1-e^{-T^p|e_p|^p})|e_p|^pdx\\
&=2-2\lambda_1\int_\Omega|e_p|^pe^{-T^p|e_p|^p}dx\\
&\geq2-\frac{2\lambda_1|\Omega|}{\mathrm e T^p}
=\frac32>1.
\end{aligned}
\end{equation}
The integral on the left is strictly increasing in $T$ and equals $1$ at $s_{e_p}$. Thus $\tau_{\mathcal A}\leq s_{e_p}<T$, uniformly in $L$. With the explicit choice \eqref{p-example-parameters}, the right-hand side of \eqref{eq b} is strictly smaller than
\[
\frac{3\lambda_1}{p}\left(\frac{T}{S_r}\right)^p
\left(\frac{3\lambda_1}{S_r^p}\right)^{\frac r{r-p}}
|\Omega|^{\frac{r-p}{r}}
=\frac{\lambda_1D_r}{p}
<\frac{\lambda_1(3+D_r)}p=\beta.
\]
This proves \eqref{eq b} without prescribing the unknown value of $\tau_{\mathcal A}$. All hypotheses of Theorem \ref{Teo 2.1} are satisfied, and $0<\Gamma(w)\leq\beta|\Omega|<\infty$ for every $w\in\partial\mathcal S_{\mathcal A}$.
\end{proof}

\medskip

The next application complements the study of superlinear elliptic equations with sign-changing weights in \cite{BZ,RQS}. Brown and Zhang \cite{BZ} investigate weighted power reactions through their Nehari geometry, whereas Ramos Quoirin and Silva \cite{RQS} analyze functionals with two degrees of homogeneity and the associated extremal parameter. Here we treat a bounded nonhomogeneous perturbation at a balanced first-eigenvalue configuration and obtain a positive ground-state level below an explicit finite boundary energy. This complements the homogeneous setting by exhibiting a different source of finite boundary barriers.

\subsection{A superlinear problem with a finite boundary barrier}\label{finite-cubic}

The main feature of this application is that $\Gamma\not\equiv+\infty$, and its finite value is realized as a limit of the reduced energy at the boundary of the spherical cone. The mechanism differs from the asymptotically $p$-linear reaction in the preceding subsection: a sign-changing cubic term cancels in a first-eigenfunction direction, where the quadratic energy also vanishes. A bounded nonhomogeneous term then determines the boundary energy. Thus the construction uses spectral degeneracy and an integral balance, while the reaction itself has cubic growth wherever its weight is nonzero.

Let $\Omega\subset\R^N$, $1\leq N\leq3$, be a bounded connected domain with smooth boundary. Let $\lambda_1<\lambda_2$ be its first two distinct Dirichlet eigenvalues and let $e_1>0$ be a first eigenfunction, normalized by $\|e_1\|=1$, where $\|u\|^2=\int_\Omega|\nabla u|^2dx$. Assume that $b\in C(\overline\Omega)\setminus\{0\}$ satisfies
\begin{equation}\label{cubic-balance}
\int_\Omega b(x)e_1^4\,dx=0.
\end{equation}
In particular, $b$ changes sign. Fix $\beta>0$ and consider, for $\mu>0$,
\begin{equation}\label{cubic-problem}\tag{$P_4$}
\begin{cases}
-\Delta u=\lambda_1u-2\beta u e^{-u^2}+\mu b(x)u^3&\text{in }\Omega,\\
u=0&\text{on }\partial\Omega.
\end{cases}
\end{equation}
For $b(x)\neq0$, the reaction $f(x,s)$ in \eqref{cubic-problem} satisfies
\[
\frac{f(x,s)}s=\lambda_1-2\beta e^{-s^2}+\mu b(x)s^2
\longrightarrow
\begin{cases}+\infty,&b(x)>0,\\-\infty,&b(x)<0,\end{cases}
\qquad\text{as}\quad |s|\to\infty.
\]
Therefore the finite barrier below does not arise from a finite asymptotic slope of the reaction.

On $X=H_0^1(\Omega)$, set
\[
Q(u)=\|u\|^2-\lambda_1\|u\|_2^2,
\qquad B(u)=\int_\Omega b(x)u^4\,dx.
\]
The associated functional is
\begin{equation}\label{cubic-energy}
I_\mu(u)=\frac12Q(u)+\beta\int_\Omega(1-e^{-u^2})\,dx-\frac\mu4B(u).
\end{equation}
The compact embedding $X\hookrightarrow L^4(\Omega)$ yields $I_\mu\in C^1(X,\R)$, and its critical points are weak solutions of \eqref{cubic-problem}. The Rayleigh characterization and simplicity of $\lambda_1$ give
\begin{equation}\label{cubic-kernel}
Q(u)\geq0,\qquad Q(u)=0\ \Longleftrightarrow\ u\in\R e_1.
\end{equation}
Define
\[
\mathcal A_b=\{u\in X:B(u)>0\},\qquad
\mathcal S_b=\mathcal A_b\cap\mathcal S,\qquad
\mathcal N_\mu=\{u\in X\setminus\{0\}:I_\mu'(u)u=0\}.
\]
The set $\mathcal A_b$ is a nonempty proper open cone: test functions supported where $b>0$ or $b<0$ give the two signs of $B$.

The following fibering calculation verifies the hypotheses needed to use the same radial parametrization as in the preceding applications.
\begin{proposition}\label{cubic-fibers}
For every $\mu>0$, $X$ and $I_\mu$ satisfy $(A_1)$--$(A_3)$ with the cone $\mathcal A_b$.
\end{proposition}
\begin{proof}
For $(A_1)$, take $\varphi(t)=t$, $\psi(u)=\|u\|^2/2$ and $J(u)v=\int_\Omega\nabla u\cdot\nabla v\,dx$. For $w\neq0$, let $\alpha_w(t)=I_\mu(tw)$. Then
\begin{equation}\label{cubic-fiber-identity}
\frac{\alpha_w'(t)}t=Q(w)+2\beta\int_\Omega w^2e^{-t^2w^2}\,dx-\mu t^2B(w).
\end{equation}
If $B(w)>0$, the right-hand side is strictly decreasing from $Q(w)+2\beta\|w\|_2^2>0$ to $-\infty$. It has a unique zero $s_w$, with the strict signs required in $(A_2)$. If $B(w)\leq0$, it is strictly positive for every $t>0$, because $w\neq0$ and $Q(w)\geq0$. Thus no ray outside $\mathcal A_b$ meets $\mathcal N_\mu$.

For $(A_3)$, the inequality $e^{-r}\geq1-r$, $r\geq0$, gives
\[
I_\mu'(u)u\geq Q(u)+2\beta\|u\|_2^2
-(2\beta+\mu\|b\|_\infty)\|u\|_4^4
\geq c_0\|u\|^2-C_\mu\|u\|^4,
\]
where $c_0=\min\{1,2\beta/\lambda_1\}>0$, by Poincar\'e's inequality. Hence $\|u\|\geq\sqrt{c_0/C_\mu}$ on $\mathcal N_\mu$, which gives the uniform lower bound for $s_w$ on $\mathcal S_b$. At the fibering maximum,
\begin{equation}\label{cubic-nehari-scale}
\mu s_w^2B(w)=Q(w)+2\beta\int_\Omega w^2e^{-s_w^2w^2}\,dx
\leq Q(w)+2\beta\|w\|_2^2.
\end{equation}
For a compact $\mathcal K\subset\mathcal S_b$, continuity gives $\min_\mathcal K B>0$, while the right-hand side is bounded. This proves the required upper bound on $\mathcal K$.
\end{proof}

Write $m_\mu(w)=s_ww$, $\Psi_\mu(w)=I_\mu(m_\mu(w))$ for $w\in\mathcal S_b$, and $c_\mu=\inf_{\mathcal N_\mu}I_\mu$. Define
\begin{equation}\label{cubic-defect}
k_\beta(s)=\beta\left[1-\left(1+\frac{s^2}{2}\right)e^{-s^2}\right].
\end{equation}
Direct calculation gives $0\leq k_\beta(s)\leq\beta$, $k_\beta(s)\to\beta$ as $|s|\to\infty$, and
\begin{equation}\label{cubic-energy-on-nehari}
I_\mu(u)-\frac14I_\mu'(u)u
=\frac14Q(u)+\int_\Omega k_\beta(u)\,dx.
\end{equation}
In particular, $I_\mu(u)>0$ on $\mathcal N_\mu$.

The next result identifies the barrier and proves that its finite value is sharp. This also shows why the choice $\Gamma\equiv+\infty$ is unavailable in this application.
\begin{proposition}\label{cubic-boundary}
The directions $\pm e_1$ belong to $\partial\mathcal S_b$. The map
\begin{equation}\label{cubic-gamma}
\Gamma_b(w)=
\begin{cases}
\beta|\Omega|,&w=\pm e_1,\\
+\infty,&w\in\partial\mathcal S_b\setminus\{\pm e_1\},
\end{cases}
\end{equation}
is lower semicontinuous and satisfies the boundary inequality in $(H)$. Moreover, for each fixed $\mu>0$ there are $w_n\in\mathcal S_b$ such that $w_n\to e_1$ and $\Psi_\mu(w_n)\to\beta|\Omega|$.
\end{proposition}
\begin{proof}
Choose $v\in X$ such that $d:=B'(e_1)v=4\int_\Omega b e_1^3v\,dx>0$. Such a choice exists since $b e_1^3\not\equiv0$. For small $\varepsilon>0$, set
\[
w_\varepsilon=\frac{e_1+\varepsilon v}{\|e_1+\varepsilon v\|}.
\]
The balance \eqref{cubic-balance} and the eigenfunction equation yield
\begin{equation}\label{cubic-boundary-perturbation}
B(w_\varepsilon)=d\varepsilon+O(\varepsilon^2)>0,
\qquad
Q(w_\varepsilon)=\frac{\varepsilon^2Q(v)}{\|e_1+\varepsilon v\|^2}=O(\varepsilon^2).
\end{equation}
Thus $e_1\in\partial\mathcal S_b$; evenness of $B$ gives the same conclusion for $-e_1$.

Now suppose $w_n\in\mathcal S_b$ and $w_n\to w\in\partial\mathcal S_b$. Continuity implies $B(w)=0$. We claim that $s_{w_n}\to\infty$. Otherwise $(A_3)$ gives a subsequence $s_{w_n}\to s\in(0,\infty)$, and continuity gives $I_\mu'(sw)(sw)=0$. This contradicts \eqref{cubic-fiber-identity} with $B(w)=0$.

If $w\neq\pm e_1$, then $Q(w)>0$ by \eqref{cubic-kernel}. Hence, using \eqref{cubic-energy-on-nehari},
\[
\Psi_\mu(w_n)\geq\frac14s_{w_n}^2Q(w_n)\longrightarrow+\infty.
\]
If $w=\pm e_1$, take a subsequence realizing the energy liminf, and then an almost everywhere convergent subsequence of $w_n$. Positivity of $e_1$ implies $|s_{w_n}w_n|\to\infty$ almost everywhere. Fatou's lemma yields
\[
\liminf_n\Psi_\mu(w_n)\geq\beta|\Omega|=\Gamma_b(w).
\]
Lower semicontinuity of \eqref{cubic-gamma} follows because its finite-value set $\{\pm e_1\}$ is closed and its finite value is constant.

Finally, take the directions $w_\varepsilon$ above. By \eqref{cubic-nehari-scale} and \eqref{cubic-boundary-perturbation}, $s_{w_\varepsilon}^2\leq C/\varepsilon$, so
\[
Q(m_\mu(w_\varepsilon))=s_{w_\varepsilon}^2Q(w_\varepsilon)\longrightarrow0.
\]
Also $s_{w_\varepsilon}\to\infty$. Along any sequence $\varepsilon\downarrow0$, an almost everywhere subsequence and dominated convergence, using $0\leq k_\beta\leq\beta$, give
\[
\int_\Omega k_\beta(m_\mu(w_\varepsilon))\,dx\longrightarrow\beta|\Omega|.
\]
Every such sequence admits this extraction, so \eqref{cubic-energy-on-nehari} gives the full limit $\Psi_\mu(w_\varepsilon)\to\beta|\Omega|$.
\end{proof}

In particular, the boundary is nonempty and the reduced energy fails to diverge there, unlike the boundary behavior in \cite[Section~3.4, Lemma~26]{SW}. Thus that boundary-divergence argument does not cover this application. Proposition \ref{cubic-boundary} also shows that $I_\mu$ is not coercive on $\mathcal N_\mu$.

Compactness below the finite barrier follows from the spectral gap, without requiring coercivity on the whole Nehari set.
\begin{proposition}\label{cubic-ps}
For every $\mu>0$, $I_\mu$ satisfies $(PS)^{\mathcal N_\mu}_c$ at each level $c<\beta|\Omega|$.
\end{proposition}
\begin{proof}
Let $u_n\in\mathcal N_\mu$, $I_\mu(u_n)\to c<\beta|\Omega|$ and $I_\mu'(u_n)\to0$ in $X'$. By \eqref{cubic-energy-on-nehari}, $Q(u_n)$ is bounded. Decompose $u_n=a_ne_1+z_n$ with $\int_\Omega z_ne_1dx=0$. The spectral gap gives
\[
Q(u_n)=Q(z_n)\geq(1-\lambda_1/\lambda_2)\|z_n\|^2,
\]
so $z_n$ is bounded. If $\|u_n\|\to\infty$ along a subsequence, then $|a_n|\to\infty$ and $u_n/a_n\to e_1$ in $X$. After an almost everywhere extraction, $|u_n|\to\infty$ almost everywhere. Therefore \eqref{cubic-energy-on-nehari} and Fatou's lemma imply $c\geq\beta|\Omega|$, a contradiction. The sequence is bounded.

Up to a subsequence, $u_n\rightharpoonup u$ in $X$ and $u_n\to u$ in $L^2(\Omega)\cap L^4(\Omega)$. The map $s\mapsto s e^{-s^2}$ is globally Lipschitz, and $u_n^3\to u^3$ in $L^{4/3}(\Omega)$. Testing $I_\mu'(u_n)$ against $u_n-u$ consequently yields
\[
\int_\Omega\nabla u_n\cdot\nabla(u_n-u)\,dx\longrightarrow0.
\]
Weak convergence then gives $\|u_n-u\|\to0$, as required.
\end{proof}

To separate the minimizing energy from the barrier, fix $w_0\in\mathcal A_b$ and define
\begin{equation}\label{cubic-threshold}
A_\beta(w_0)=\frac12Q(w_0)+\beta\|w_0\|_2^2,
\qquad
\mu_*(w_0)=\frac{A_\beta(w_0)^2}{\beta|\Omega|B(w_0)}.
\end{equation}
Since $1-e^{-r}\leq r$ for $r\geq0$, Proposition \ref{Prop 1.2}(d) gives
\begin{equation}\label{cubic-gap}
c_\mu\leq\max_{t>0}I_\mu(tw_0)
\leq\max_{t>0}\left(A_\beta(w_0)t^2-\frac\mu4B(w_0)t^4\right)
=\frac{A_\beta(w_0)^2}{\mu B(w_0)}.
\end{equation}
Thus $\mu>\mu_*(w_0)$ guarantees the strict separation in $(H)$.

\begin{theorem}\label{cubic-ground-state}
Suppose \eqref{cubic-balance} holds and $\beta>0$. Fix any $w_0\in\mathcal A_b$. If $\mu>\mu_*(w_0)$, problem \eqref{cubic-problem} has a ground-state solution $u_\mu$ satisfying
\[
0<I_\mu(u_\mu)=c_\mu
=\inf_{w\in\mathcal A_b}\max_{t>0}I_\mu(tw)
<\beta|\Omega|=\inf_{\partial\mathcal S_b}\Gamma_b.
\]
\end{theorem}
\begin{proof}
Propositions \ref{cubic-fibers}, \ref{cubic-boundary} and \ref{cubic-ps}, together with \eqref{cubic-gap}, verify all hypotheses of Theorem \ref{level}. Every nonzero critical point belongs to $\mathcal N_\mu$, so the resulting minimizer is a ground state.
\end{proof}

The balance condition has an explicit realization on a ball.
\begin{example}\label{cubic-ball}
Let $\Omega=B_1(0)\subset\R^3$ and $b(x)=x_1$. The first eigenfunction is radial, so \eqref{cubic-balance} follows by symmetry. For $0<\varepsilon<1$, let $w_0=(1+\varepsilon x_1)e_1$. Then
\[
B(w_0)=4\varepsilon\int_\Omega x_1^2e_1^4\,dx
+4\varepsilon^3\int_\Omega x_1^4e_1^4\,dx>0.
\]
Consequently Theorem \ref{cubic-ground-state} applies whenever $\mu>\mu_*(w_0)$. Here $f(x,s)/s$ has no finite limit at infinity for almost every $x\in\Omega$, while the finite boundary energy is $\beta|\Omega|$.
\end{example}

The next application complements related studies of quasilinear equations with indefinite weights. Omari and Sovrano \cite{OSlog,OSsuper} treat the mean-curvature operator, with logistic and superlinear reactions, respectively. Qi and Zhang \cite{QZ} use Nehari fibering for nonhomogeneous energies depending jointly on $u$ and $\nabla u$. Our operator has a different structure: a uniformly elliptic Laplacian is perturbed by a bounded, gradient-dependent energy at a balanced first-eigenvalue configuration. The result below concerns this configuration and an explicit finite boundary threshold; it does not assert an extension of the theorems for those different operators.

\subsection{A quasilinear problem with a geometric finite barrier}\label{gradient-barrier}

Here again $\Gamma\not\equiv+\infty$, but its finite value is no longer $\beta|\Omega|$ or a constant saturation value times a support measure. Instead, it retains the spatial profile of a first eigenfunction and its gradient. The spectral balance from Subsection~\ref{finite-cubic} is combined with a bounded perturbation involving the gradient. The cubic reaction remains superlinear wherever $b\neq0$; the finite barrier arises from this gradient-dependent saturation, rather than from an asymptotically linear reaction.

Retain $\Omega$, $\lambda_1<\lambda_2$, $e_1$, $Q$, $B$ and $\mathcal A_b$ from Subsection~\ref{finite-cubic}, including \eqref{cubic-balance}, and fix $0<\beta<2$. On $X=H_0^1(\Omega)$ set
\begin{equation}\label{gradient-energy}
\mathcal I_\mu(u)=\frac12Q(u)+\int_\Omega h_\beta(u,\nabla u)\,dx-\frac\mu4B(u),
\qquad h_\beta(s,\xi)=\frac{\beta|\xi|^2}{1+s^2+|\xi|^2}.
\end{equation}
Writing $D=1+s^2+|\xi|^2$, its derivatives are
\begin{equation}\label{gradient-derivatives}
(h_\beta)_s=-\frac{2\beta s|\xi|^2}{D^2},\qquad
(h_\beta)_\xi=\frac{2\beta(1+s^2)\xi}{D^2}.
\end{equation}
The estimates $|(h_\beta)_s|\leq\beta/4$ and $|(h_\beta)_\xi|\leq\beta$ show that both derivatives are bounded and continuous. Dominated convergence for these derivatives, together with $X\hookrightarrow L^4(\Omega)$, gives $\mathcal I_\mu\in C^1(X,\R)$. Its critical points solve
\begin{equation}\label{gradient-problem}\tag{$P_5$}
\begin{cases}
-\div\left[\left(1+\dfrac{2\beta(1+u^2)}{D_u^2}\right)\nabla u\right]
=\lambda_1u+\dfrac{2\beta u|\nabla u|^2}{D_u^2}+\mu b(x)u^3
&\text{in }\Omega,\\
u=0&\text{on }\partial\Omega,
\end{cases}
\qquad D_u=1+u^2+|\nabla u|^2.
\end{equation}
The full flux is not bounded: only the energy perturbation saturates. To verify uniform ellipticity, put
$A_\beta(s,\xi)=\xi+(h_\beta)_\xi(s,\xi)$ and $c=1+s^2$. The tangential and radial eigenvalues of $(h_\beta)_{\xi\xi}$ are, respectively,
\[
\frac{2\beta c}{(c+|\xi|^2)^2},\qquad
\frac{2\beta c(c-3|\xi|^2)}{(c+|\xi|^2)^3}.
\]
Since $(1-3y)/(1+y)^3\geq-1/4$ for $y\geq0$, these give
\begin{equation}\label{gradient-monotonicity}
\big(A_\beta(s,\xi)-A_\beta(s,\eta)\big)\cdot(\xi-\eta)
\geq(1-\beta/2)|\xi-\eta|^2.
\end{equation}

\begin{proposition}\label{gradient-fibers}
For each $\mu>0$, $X$ and $\mathcal I_\mu$ satisfy $(A_1)$--$(A_3)$ with the cone $\mathcal A_b$.
\end{proposition}
\begin{proof}
The Hilbert-space verification of $(A_1)$ is unchanged. Put $R_w=w^2+|\nabla w|^2$. For $\alpha_w(t)=\mathcal I_\mu(tw)$,
\begin{equation}\label{gradient-fiber}
\frac{\alpha_w'(t)}t=Q(w)+2\beta\int_\Omega
\frac{|\nabla w|^2}{(1+t^2R_w)^2}\,dx-\mu t^2B(w).
\end{equation}
For $B(w)>0$ this is strictly decreasing from $Q(w)+2\beta\|w\|^2>0$ to $-\infty$. If $B(w)\leq0$ and $w\neq0$, it is positive, since $Q(w)\geq0$ and $\nabla w\not\equiv0$. This proves $(A_2)$.

For the uniform lower bound in $(A_3)$, suppose $w_n\in\mathcal S_b$ and $s_n=s_{w_n}\to0$. Dividing the Nehari identity by $s_n^2$ gives
\begin{equation}\label{gradient-small-scale}
Q(w_n)+2\beta\int_\Omega\frac{|\nabla w_n|^2}{(1+s_n^2R_{w_n})^2}\,dx
=\mu s_n^2B(w_n)\longrightarrow0.
\end{equation}
Here $B(w_n)$ is bounded by the $L^4$ embedding. Hence $Q(w_n)\to0$. By the spectral gap and $\|w_n\|=1$, a subsequence converges strongly in $X$ to $w=\pm e_1$. After an almost everywhere extraction for both functions and gradients, the integral in \eqref{gradient-small-scale} tends to $\|w\|^2=1$: replace $|\nabla w_n|^2$ by $|\nabla w|^2$ with an $L^1$ error tending to zero, then use dominated convergence. This contradicts \eqref{gradient-small-scale}. Finally, \eqref{gradient-fiber} at its zero yields
\begin{equation}\label{gradient-scale-bound}
\mu s_w^2B(w)\leq Q(w)+2\beta\|w\|^2.
\end{equation}
The positive minimum of $B$ on each compact subset of $\mathcal S_b$ gives the upper bound in $(A_3)$.
\end{proof}

Denote the Nehari set, radial map and reduced energy by $\mathcal N_\mu^\nabla$, $m_\mu^\nabla$ and $\Psi_\mu^\nabla$, respectively. Direct differentiation of \eqref{gradient-energy} gives
\begin{equation}\label{gradient-defect}
\begin{aligned}
\mathcal I_\mu(u)-\tfrac14\mathcal I_\mu'(u)u
&=\tfrac14Q(u)+\int_\Omega k_\beta(u,\nabla u)\,dx,\\
k_\beta(s,\xi)&=\frac{\beta|\xi|^2(\frac12+s^2+|\xi|^2)}{(1+s^2+|\xi|^2)^2},
\qquad 0\leq k_\beta\leq\beta.
\end{aligned}
\end{equation}
The finite threshold is
\begin{equation}\label{gradient-threshold}
G_\beta=\beta\int_\Omega\frac{|\nabla e_1|^2}{e_1^2+|\nabla e_1|^2}\,dx,
\qquad 0<G_\beta<\beta|\Omega|.
\end{equation}
The strict inequalities follow from $e_1>0$ almost everywhere and $\|e_1\|=1$.

\begin{proposition}\label{gradient-boundary-ps}
The lower semicontinuous map
\begin{equation}\label{gradient-gamma}
\Gamma_\nabla(w)=
\begin{cases}G_\beta,&w=\pm e_1,\\+\infty,&w\in\partial\mathcal S_b\setminus\{\pm e_1\},\end{cases}
\end{equation}
satisfies the boundary inequality in $(H)$, and its finite value is attained as a boundary limit of $\Psi_\mu^\nabla$. Moreover, $\mathcal I_\mu$ satisfies $(PS)^{\mathcal N_\mu^\nabla}_c$ for every $c<G_\beta$.
\end{proposition}
\begin{proof}
Let $w_n\to w\in\partial\mathcal S_b$ strongly in $X$, with $w_n\in\mathcal S_b$. Since $B(w)=0$, $(A_3)$ and the strict positivity in \eqref{gradient-fiber} outside the cone imply $s_n=s_{w_n}\to\infty$, exactly as in Proposition~\ref{cubic-boundary}. If $w\neq\pm e_1$, \eqref{gradient-defect} gives $\Psi_\mu^\nabla(w_n)\geq s_n^2Q(w_n)/4\to\infty$. If $w=\pm e_1$, extract almost everywhere convergence of $w_n$ and $\nabla w_n$ along a subsequence realizing the energy liminf. Since $e_1>0$ almost everywhere,
\[
k_\beta(s_nw_n,s_n\nabla w_n)
\longrightarrow\beta\frac{|\nabla e_1|^2}{e_1^2+|\nabla e_1|^2}
\quad\text{almost everywhere}.
\]
Fatou's lemma proves the boundary inequality. Lower semicontinuity follows from the closed finite-value set $\{\pm e_1\}$.

For sharpness, use the perturbations $w_\varepsilon$ in \eqref{cubic-boundary-perturbation}. Then $B(w_\varepsilon)\sim d\varepsilon$ and $Q(w_\varepsilon)=O(\varepsilon^2)$. Equation~\eqref{gradient-scale-bound} gives $s_{w_\varepsilon}^2=O(\varepsilon^{-1})$, so $Q(m_\mu^\nabla(w_\varepsilon))\to0$. The preceding pointwise limit and $k_\beta\leq\beta$ give $\Psi_\mu^\nabla(w_\varepsilon)\to G_\beta$ by dominated convergence.

Now let $u_n\in\mathcal N_\mu^\nabla$, $\mathcal I_\mu(u_n)\to c<G_\beta$ and $\mathcal I_\mu'(u_n)\to0$ in $X'$. By \eqref{gradient-defect}, $Q(u_n)$ is bounded. Write $u_n=a_ne_1+z_n$, with $\int_\Omega z_ne_1dx=0$. The spectral estimate $Q(z_n)\geq(1-\lambda_1/\lambda_2)\|z_n\|^2$ bounds $z_n$. If $u_n$ were unbounded, a subsequence would satisfy $|a_n|\to\infty$ and $u_n/a_n\to e_1$ strongly in $X$. Almost everywhere extraction for functions and gradients, followed by Fatou's lemma in \eqref{gradient-defect}, would give $c\geq G_\beta$, a contradiction.

Thus, after extraction, $u_n\rightharpoonup u$ in $X$ and $u_n\to u$ in $L^2\cap L^4$ and almost everywhere. Testing the derivative against $u_n-u$, the boundedness of $(h_\beta)_s$ and the strong $L^2,L^4$ convergence give
\[
\int_\Omega A_\beta(u_n,\nabla u_n)\cdot\nabla(u_n-u)\,dx\longrightarrow0.
\]
Moreover, the bounded correction $(h_\beta)_\xi$ in \eqref{gradient-derivatives} yields $A_\beta(u_n,\nabla u)\to A_\beta(u,\nabla u)$ in $L^2$ by dominated convergence. Its pairing with $\nabla(u_n-u)$ therefore tends to zero by weak convergence. Subtracting and using \eqref{gradient-monotonicity} gives
\[
(1-\beta/2)\|u_n-u\|^2
\leq\int_\Omega\big[A_\beta(u_n,\nabla u_n)-A_\beta(u_n,\nabla u)\big]
\cdot\nabla(u_n-u)\,dx\longrightarrow0.
\]
This proves the required strong compactness.
\end{proof}

For $w_0\in\mathcal A_b$, set $A_\beta^\nabla(w_0)=Q(w_0)/2+\beta\|w_0\|^2$ and
$\mu_*^\nabla(w_0)=A_\beta^\nabla(w_0)^2/(G_\beta B(w_0))$.
Since $h_\beta(tw_0,t\nabla w_0)\leq\beta t^2|\nabla w_0|^2$, the radial minimax characterization gives
\[
c_\mu^\nabla\leq\max_{t>0}\mathcal I_\mu(tw_0)
\leq\max_{t>0}\left[A_\beta^\nabla(w_0)t^2-\tfrac\mu4B(w_0)t^4\right]
=\frac{A_\beta^\nabla(w_0)^2}{\mu B(w_0)}.
\]
We consequently obtain the following application of the finite-barrier alternative.
\begin{theorem}\label{gradient-ground-state}
Assume \eqref{cubic-balance} and $0<\beta<2$. For any $w_0\in\mathcal A_b$ and $\mu>\mu_*^\nabla(w_0)$, problem~\eqref{gradient-problem} has a ground-state solution with
\[
0<\mathcal I_\mu(u_\mu)=c_\mu^\nabla
=\inf_{w\in\mathcal A_b}\max_{t>0}\mathcal I_\mu(tw)
<G_\beta=\inf_{\partial\mathcal S_b}\Gamma_\nabla.
\]
\end{theorem}
\begin{proof}
Propositions~\ref{gradient-fibers} and \ref{gradient-boundary-ps}, together with the preceding estimate, verify the hypotheses of Theorem~\ref{level}. All nonzero critical points belong to $\mathcal N_\mu^\nabla$.
\end{proof}

The ball and the weight in Example~\ref{cubic-ball} also give an explicit admissible configuration here. On the interval $(0,L)$, the profile $e_1(x)=C\sin(\pi x/L)$ gives $G_\beta=\beta\pi L/(L+\pi)$, by substituting $y=\pi x/L$ and then $t=\tan y$ on $(0,\pi/2)$. In general \eqref{gradient-threshold} retains the relative sizes of $e_1$ and $\nabla e_1$ throughout the domain. The sharp finite boundary limit also prevents use of the boundary-divergence argument in \cite[Section~3.4, Lemma~26]{SW}, as in the preceding application.

The Theorem \ref{main} of next section complements Theorem 1.1 in \cite{PZ} and Corollary 4.2(1) in \cite{FQS}.

\subsection{A Kirchhoff problem of Perera-Zhang type}\label{kirchhoff-app}

By using the Yang index and critical groups, Perera and Zhang in \cite{PZ} proved the existence of nontrivial solutions for the following class of nonlocal quasilinear elliptic boundary value
problem
\begin{equation}                                \label{P'_9}\tag{$P_6$}
\left\{\begin{array}{rcl}
-(a+b\displaystyle\int_{\Omega}\vert \nabla u\vert^2 \ dx)\Delta u&=&f(x, u) ~~\mbox{in}~~\Omega,\\[0.3cm]
u&=& 0 ~~\mbox{on}~~\partial\Omega,
\end{array}
\right.
\end{equation}
where $\Omega$ is a smooth bounded domain $\mathbb{R}^N$, $N=1,\, 2$ or $3$, $a,b>0$ are constants and $f \in C(\Omega\times\R, \R)$ is a Carathéodory function, satisfying:
\begin{equation}            \label{f'_{18}}\tag{$PZ$}
\lim_{s\to 0}\dfrac{f(x, s)}{as}=\lambda \ \ \ \mbox{and} \ \ \ \lim_{|s|\to\infty}\dfrac{f(x, s)}{bs^3}=\mu,~~\mbox{uniformly in}~~x\in\Omega.
\end{equation}
To prove their result, the authors first consider the following nonlinear problem

\begin{equation}             \label{P'_{11}}\tag{$Q_{\mu}$}
\left\{\begin{array}{rcl}
-\|u\|^2\Delta u&=&\mu u^3~~\mbox{in}~~\Omega,\\[0.2cm]
u&=& 0~~\mbox{on}~~\partial\Omega.
\end{array}
\right.
\end{equation}
Denoting the sequence $\{\mu_n\}$ of minimax eigenvalues of \eqref{P'_{11}} by $0<\mu_1\leq\mu_2\leq\cdots$ and by $0 < \lambda_1 < \lambda_2 \leq \cdots$ the eigenvalues of $(-\Delta, H_0^1(\Omega))$, they were able to prove that, under the condition (\ref{f'_{18}}), if $\lambda\in(\lambda_l,\lambda_{l+1})$ and $\mu\in(\mu_m,\mu_{m+1})$ is not an eigenvalue of (\ref{P'_{11}}), with $l\neq m$, then problem \eqref{P'_9} has a nontrivial solution. 

More recently, Figueiredo et al. in \cite{FQS} studied the particular case of \eqref{P'_9}:
\begin{equation}                                \label{pl}\tag{$PL$}
\left\{\begin{array}{rcl}
-(a+b\displaystyle\int_{\Omega}\vert \nabla u\vert^2 \ dx)\Delta u&=&\lambda u+\mu |u|^2u, ~~\mbox{in}~~\Omega,\\[0.3cm]
u&=& 0, ~~\mbox{on}~~\partial\Omega,
\end{array}
\right.
\end{equation}
they showed the existence of a ground-state solutions for many ranges of the pair $(\lambda,\mu)$, in particular for the case $\lambda<\lambda_1 a$ and $\mu>b\mu_1$, where
$$
\mu_1=\inf\left\{\dfrac{\|u\|^4}{|u|_4^4}:~ u\in H_0^1(\Omega)\setminus\{0\}\right\}.
$$

In this section, we will study the existence of a solution to the Kirchhoff problem with asymptotically cubic nonlinearity
\begin{equation}\tag{$P_6$}\label{Prop P_4}
\left\{\begin{array}{rcl}
-(a+b\|u\|^2)\Delta u&=&f(u)~~\mbox{in}~~\Omega,\\[0.3cm]
u&=& 0~~\mbox{on}~~\partial\Omega,
\end{array}
\right.
\end{equation}
where $\Omega \subset \R^N$, $1\leq N\leq3$, is a bounded domain with smooth boundary, $a, b>0$ and $f \in C(\R)$ satisfies:
\begin{itemize}
    \item[$(f_1)$]
$\displaystyle\lim_{s\to 0}\dfrac{f(s)}{as}=\lambda \ \ \ \mbox{and} \ \ \ \lim_{|s|\to\infty}\dfrac{f(s)}{bs^3}=\mu$;

\item[$(f_2)$] The map $s\mapsto (f(s)-a\lambda s)/|s|^3$ is nondecreasing in $s\neq 0$;

\end{itemize}

As an immediate consequence of $(f_2)$ we have
\begin{equation}\label{a33}
   f(s)s\geq 4F(s)-a\lambda s^2~~\mbox{for all}~~ s\in\mathbb{R},
\end{equation}
where $F(s)=\displaystyle\int_0^sf(\tau)d\tau$. Quotients involving $f(sw)/(sw)^3$ below are integrated over $[w\neq0]$; their products with $w^4$ are interpreted as $f(sw)w/s^3$, equal to zero on $[w=0]$. Indeed, for $r(s)=f(s)-a\lambda s$, condition $(f_2)$ implies
\[
r(s)s\geq4\int_0^sr(t)dt.
\]
For $s>0$ this follows by integrating $r(t)/t^3\leq r(s)/s^3$ over $0<t<s$. For $s<0$, the substitution $t=-\rho$ gives the same inequality. Expanding $r$ yields \eqref{a33}.

\begin{example}
  The functions $f(s)=a\lambda s+b\mu s^3$ and $f(s)=a\lambda s+\frac{2b\mu}{\pi} s^3\mathrm{arctg}|s|$, for $s\in\mathbb{R}$, satisfy $(f_1)-(f_2)$.  
\end{example}

We start by defining the energy functional $I: H^{1}_{0}(\Omega) \rightarrow \R$, associated to the problem (\ref{Prop P_4}), as
$$
I(u)=\dfrac{a}{2}\,\|u\|^2+\dfrac{b}{4}\,\|u\|^4-\int_{\Omega}F(u)dx.
$$
It is well known that $I \in C^{1}(H^{1}_{0}(\Omega),\R)$ and 
$$
I'(u)v = (a+b\|u\|^2)\int_{\Omega}\nabla u\nabla vdx - \int_{\Omega}f(u)v,~~\forall v\in H_0^1(\Omega).
$$

The Nehari set associated with the functional $I$ is
$$
\mathcal{N}:=\left\{u\in H_0^1(\Omega)\setminus\{0\}:~ a\|u\|^2+b\|u\|^4=\int_\Omega f(u)udx\right\}.
$$
For our purposes we consider the sets
$$
\mathcal{A}_\mu:=\left\{u\in H_0^1(\Omega):~ \|u\|^4<\mu\int_\Omega u^4dx\right\}.
$$
and $\mathcal{S}_\mu:=\mathcal{S}\cap \mathcal{A}_\mu$ where $\mathcal{S}$ denotes the unit sphere of $H^{1}_{0}(\Omega)$, that is,
\[
\mathcal{S}_\mu = \left\{u\in \mathcal{S}:~ 1<\mu\int_\Omega u^4dx\right\}.
\]


\begin{lemma}                                         \label{Lema 5.1}
Suppose that $\mu>\mu_1$. Then, $\mathcal{A}_{\mu}$ is a nonempty open set and $s\mathcal{A}_{\mu}\subset \mathcal{A}_{\mu}$, for all $s > 0$.
\end{lemma}
\begin{proof}
The continuous map $h(u)=\|u\|^4-\mu|u|_4^4$ is positively homogeneous of degree four. Hence $\mathcal A_\mu=h^{-1}(( -\infty,0))$ is an open cone. Since $\mu>\mu_1=\inf_{u\neq0}\|u\|^4/|u|_4^4$, there is $u\neq0$ with $h(u)<0$, proving nonemptiness.
\end{proof}

\begin{proposition}                   \label{Lema 5.2}
Suppose that $f$ satisfies $(f_1)-(f_2)$. If $\lambda\in [0, \lambda_1)$ and $\mu>\mu_1$, then $(A_1)-(A_3)$ hold.
\end{proposition}
\begin{proof}
The Hilbert-space verification of $(A_1)$ is as in Proposition~\ref{Prop 3.1}. Put $H(t)=[f(t)-a\lambda t]/t^3$ for $t\neq0$ and $D(w)=\|w\|^2-\lambda|w|_2^2$. By $(f_2)$, $H(sw)$ is nondecreasing in $s$ for nonzero real $w$, $H\leq b\mu$, and $D(w)\geq(1-\lambda/\lambda_1)\|w\|^2>0$. Thus
\begin{equation}\label{a}
\frac{\alpha_w'(s)}{s^3}=\frac{aD(w)}{s^2}+b\|w\|^4-\int_{[w\neq0]}H(sw)w^4dx
\end{equation}
is strictly decreasing. The integral is read as $[f(sw)-a\lambda sw]w/s^3$, and the growth bound $|f(t)|\leq C(|t|+|t|^3)$ ensures integrability. Dominated convergence and $(f_1)$ give
\[
\lim_{s\to0}\frac{\alpha_w'(s)}s=aD(w)>0,\qquad
\lim_{s\to\infty}\frac{\alpha_w'(s)}{s^3}=b(\|w\|^4-\mu|w|_4^4).
\]
For $w\in\mathcal A_\mu$ the latter limit is negative, so \eqref{a} has exactly one zero with the signs in $(A_2)$. Outside the cone, $H\leq b\mu$ gives $\alpha_w'(s)/s^3\geq aD(w)/s^2>0$.

For $(A_3)$, the estimate $|f(t)-a\lambda t|\leq\varepsilon|t|+C_\varepsilon|t|^3$ gives
\[
I'(u)u\geq\left[a(1-\lambda/\lambda_1)-\varepsilon/\lambda_1\right]\|u\|^2-C_\varepsilon'\|u\|^4.
\]
Choosing $0<\varepsilon<a(\lambda_1-\lambda)$ bounds Nehari norms away from zero. If $w_n$ lies in a compact $\mathcal K\subset\mathcal S_\mu$ and $s_n=s_{w_n}\to\infty$, extract $w_n\to w\in\mathcal K$ strongly in $H_0^1$ and $L^4$. The Nehari identity gives
\begin{equation}\label{eq3}
\frac a{s_n^2}+b=\int_\Omega\frac{f(s_nw_n)w_n}{s_n^3}dx.
\end{equation}
The integrands tend almost everywhere to $b\mu w^4$ and are bounded by $C(s_n^{-2}|w_n|^2+|w_n|^4)$. A dominated subsequence, supplied by strong $L^4$ convergence, gives $1=\mu|w|_4^4$, contradicting $w\in\mathcal S_\mu$.
\end{proof}


\begin{lemma}                  \label{Prop 5.1}
Suppose that $f$ satisfies $(f_1)-(f_2)$, $\lambda\in [0, \lambda_1)$ and $\mu>\mu_1$. If $\left\{w_n\right\}\subset \mathcal{S}_\mu$ is such that
\begin{equation}\label{d}
  \mathrm{dist}(w_n,\partial \mathcal{S}_\mu)\to 0,~~\mbox{as}~~n\to\infty.  
\end{equation}
Then, after a subsequence, $w_n\rightharpoonup w$ in $H_0^1(\Omega)$ for some $w\neq0$ satisfying $\mu|w|_4^4=1$ and $\|w\|\leq1$. Moreover, $s_n:=s_{w_n}\to\infty$ and $\Psi(w_n)\to+\infty$ as $n\to\infty$.
\end{lemma}
\begin{proof} The sequence $\left\{w_n\right\}$ is bounded in $H_0^1(\Omega)$, because  $(w_n)\subset \mathcal{S}_\mu$. Passing to a subsequence, there exists $w\in H_0^1(\Omega)$ such that $w_n\rightharpoonup w$ in $H_0^1(\Omega)$. It follows from \eqref{d} that there exists a sequence $\left\{z_n\right\}\subset\partial \mathcal{S}_\mu$ satisfying $\|w_n-z_n\|\to 0$ as $n\to\infty$.

\medskip

The fact that $z_n\in\partial \mathcal{S}_\mu$ implies
$$
1=\mu\int_\Omega z_n^4dx.
$$

Now, we can argue as in the proof of Lemma \ref{Prop 2.2} in order to show that
\begin{equation}\label{s8}
 \left|\mu\int_\Omega w_n^4dx-1\right|\leq C \|w_n-z_n\|,
\end{equation}
for some $C>0$. Passing to the limit as $n\to\infty$ in \eqref{s8}, using Sobolev compact embedding of $H_0^1(\Omega)$ into $L^4(\Omega)$ and the convergence $\|w_n-z_n\|\to 0$ we obtain
 $$
 \left|\mu\int_\Omega w^4dx-1\right|=0,
 $$
 Thus $\mu|w|_4^4=1$ and $w\neq0$, while weak lower semicontinuity gives only $\|w\|\leq1$. 

\medskip

Suppose by contradiction that there exists a bounded subsequence of $\left\{s_{n}\right\}$, still denoted by $\left\{s_{n}\right\}$. Passing to a subsequence if necessary, we can assume that $\left\{s_{n}\right\}$ converges to a real number $s_0$, which is positive by Proposition \ref{Lema 5.2}. Since $w_n\in S_\mu$, we obtain
$$
\dfrac{a}{s_n^2}+b=\int_{[w_n\neq 0]}\dfrac{f(s_nw_n)}{(s_nw_n)^3}w_n^4 dx.
$$
The growth bound $|f(t)|\leq C(|t|+|t|^3)$ and compact $L^4$ convergence justify convergence of $f(s_nw_n)w_n/s_n^3$ to $f(s_0w)w/s_0^3$. Letting $n\to\infty$, it follows that
\begin{equation}\label{b}
    \dfrac{a}{s_0^2}+b\leq\int_{[w\neq 0]}\dfrac{f(s_0w)}{(s_0w)^3}w^4 dx.
\end{equation}
By hypothesis $(f_2)$, inequality \eqref{b} implies
$$
\dfrac{a}{s_0^2}+b\leq\int_{[w\neq 0]}\dfrac{f(s_0w)-a\lambda s_0w+a\lambda s_0 w}{(s_0w)^3}w^4 dx\leq b\mu\int_\Omega w^4dx+\dfrac{a\lambda}{s_0^2}\int_\Omega w^2
$$
or equivalently
\begin{equation}\label{b1}
    b\left(1-\mu\int_{\Omega}w^4dx\right)\leq \dfrac{a}{s_0^2}\left(\lambda\int_{\Omega} w^2dx-1\right),
\end{equation}
but since $\mu|w|_4^4=1$, the left hand side of \eqref{b1} is equal to zero, while its right hand side is strictly less than zero because $\lambda<\lambda_1$ and $\|w\|\leq1$, leading us to a contradiction, therefore $s_n\to\infty$. Finally, note that
$$
\Psi(w_n)=I(s_n w_n)-\frac{1}{4}I'(s_n w_n)(s_nw_n)=\frac{as_n^2}{4}+\int_{\Omega}\left[\frac{1}{4}f(s_n w_n)s_n w_n-F(s_n w_n)\right] dx
$$
It is a consequence of \eqref{a33} that
$$
\Psi(w_n)\geq \frac{as_n^2}{4}\left(1-\lambda\int_{\Omega}w_n^2 dx\right)\geq \frac{as_n^2}{4}\left(1-\frac{\lambda}{\lambda_1}\right).
$$
Since $\lambda\in [0, \lambda_1)$, we conclude that $\Psi(w_n)\to+\infty$ as $n\to+\infty$.
\end{proof}

\begin{proposition}\label{boto}
Suppose that $f$ satisfies $(f_1)-(f_2)$. If $\lambda\in [0, \lambda_1)$ and $\mu>\mu_1$, then the functional $I$ verifies hypothesis $(H)$.
\end{proposition}
\begin{proof}
By Lemma \ref{Prop 5.1}, it is enough to take $\Gamma\equiv +\infty$. This application also uses a proper spherical cone, here with an infinite boundary barrier.
\end{proof}

\begin{proposition}                     \label{psc}
Suppose that $f$ satisfies $(f_1)-(f_2)$. If $\lambda\in [0, \lambda_1)$ and $\mu>\mu_1$, then the functional $\Psi$ satisfies the Palais-Smale condition at the level $c$ on $\mathcal{S}_\mu$, for any $c\in [c_{\mathcal{N}}, +\infty)$.
\end{proposition}
\begin{proof} A PS sequence for $\Psi$ gives an ambient PS sequence in $\mathcal N$ by Proposition \ref{Abstract}. It therefore suffices first to prove $(PS)^{\mathcal N}_c$ for $I$; the convergence of the normalized vectors then follows from the continuous inverse $m^{-1}$. Let $\left\{u_n\right\}\subset \mathcal{N}$ be a $(PS)_c$ sequence for $I$. Let us first prove that $\{u_n\}$ is bounded in $H_0^1(\Omega)$. In fact, in the contrary, up to a subsequence still denoted by $\left\{u_n\right\}$, we have $\|u_n\|\to\infty$ as $n\to\infty$. Set
$$
v_n:=\dfrac{u_n}{\|u_n\|}=m^{-1}(u_n)\in \mathcal{S}_\mu.
$$
It follows that the sequence $\left\{v_n\right\}$ is bounded and 
\begin{equation}\label{jadeu}
\Psi(v_n)\to c \ \mbox{as $n\to\infty$}.
\end{equation}
Consequently, there exists $v\in H_0^1(\Omega)$ such that $v_n\rightharpoonup v$ in  $H_0^1(\Omega)$. We claim that $v\neq 0$. In fact, suppose by contradiction that $v=0$. Since $\left\{\Psi(v_n)\right\}$ is bounded, there exists $C>0$ such that
$$
C\geq\Psi(v_n)=I(s_nv_n)\geq I(sv_n)=\dfrac{as^2}{2}+\dfrac{bs^4}{4}-\int_\Omega F(sv_n)dx,
$$
for all $s>0$. Once $v_n\rightharpoonup 0$ in $H_0^1(\Omega)$, passing to the limit as $n\to\infty$, we get
$$
C\geq\dfrac{as^2}{2}+\dfrac{bs^4}{4}~~\mbox{for all}~~s>0,
$$
leaving us to a contradiction. Therefore $v\neq 0$. 

\medskip

Since $\left\{u_n\right\}\subset\mathcal{N}$ is a $(PS)_c$ sequence for the functional $I$, we have
$$
o_n(1)+(a+b\|u_n\|^2)\|u_n\|\int_{\Omega}\nabla v_n\nabla w dx=\|u_n\|^3\int_{\Omega}\dfrac{f(\|u_n\|v_n)}{(\|u_n\|v_n)^3}v_n^3wdx,
$$
for all $w\in H_0^1(\Omega)$. Dividing the above expression by $\|u_n\|^3$, we get
$$
o_n(1)+\left(\dfrac{a}{\|u_n\|^2}+b\right)\int_{\Omega}\nabla v_n\nabla w dx=\int_{\Omega}\dfrac{f(\|u_n\|v_n)}{(\|u_n\|v_n)^3}v_n^3wdx,
$$
Set $t_n=\|u_n\|$. The nonlinear source in the last equation is $f(t_nv_n)/t_n^3$, with no singular quotient at zeros. The growth bound gives
\[
\left|\frac{f(t_nv_n)}{t_n^3}\right|
\leq C\bigl(t_n^{-2}|v_n|+|v_n|^3\bigr).
\]
By compact convergence $v_n\to v$ in $L^4$, these sources converge almost everywhere to $b\mu v^3$, and their $4/3$ powers are uniformly integrable. Vitali's theorem gives convergence in $L^{4/3}$. The coefficient on the linear gradient term tends to $b$. Passing to the limit using this source convergence and the weak convergence of the gradients, we obtain
\begin{equation}\label{eq4}
\int_{\Omega}\nabla v\nabla w=\mu\int_{\Omega}v^3wdx.
\end{equation}

\medskip

Testing the normalized PS equation against $v_n$ and using the strong $L^4$ source convergence just established, together with \eqref{eq4} tested against $v$, gives  
$$
b\|v\|^2=b\mu\int_{\Omega} v^{4} dx=\lim_{n\to+\infty}\int_{\Omega}\left[\frac{f(\|u_n\|v_n)}{(\|u_n\|v_n)^3}\right]v_n^4 dx=b\lim_{n\to\infty}\|v_n\|^2.
$$
This shows that $v_n\to v$ in $H_{0}^{1}(\Omega)$ and $\|v\|=1$. Combining this information and \eqref{eq4}, we get
$$
1=\|v\|^{2}=\mu\int_{\Omega}v^4 dx.
$$
Thus $v\in \partial\mathcal{S}_{\mathcal{A}}$. Now, Lemma \ref{Prop 5.1} implies $\Psi(v_n)\to+\infty$, which is a clear contradiction with \eqref{jadeu}. Therefore, $\left\{u_n\right\}$ is a bounded sequence.

\medskip

For the bounded sequence, extract $u_n\rightharpoonup u$ in $H_0^1$ and $u_n\to u$ in $L^4$. The cubic growth bound gives $f(u_n)\to f(u)$ in $L^{4/3}$. Since $a+b\|u_n\|^2$ is bounded and at least $a>0$, testing the derivative against $u_n-u$ yields
\[
a\|u_n-u\|^2\leq(a+b\|u_n\|^2)\|u_n-u\|^2
=I'(u_n)(u_n-u)+\int_\Omega f(u_n)(u_n-u)dx+o(1)\longrightarrow0.
\]
Here weak convergence gives the $o(1)$ pairing with $u$. This proves $(PS)^{\mathcal N}_c$ for $I$ and, by the radial correspondence stated at the beginning, $(PS)_c$ for $\Psi$.
\end{proof} 

Our main result in this subsection is as follows:

\begin{theorem}                                                   \label{main}
Suppose that $f$ satisfies $(f_1)-(f_2)$. If $\lambda\in [0, \lambda_1)$ and $\mu>\mu_1$, then the problem \eqref{Prop P_4} has a least energy nontrivial solution. 
\end{theorem}

\begin{proof}  
It is a consequence of Lemma \ref{Lema 5.1}, Propositions \ref{Lema 5.2}, \ref{boto}, \ref{psc} and Theorem \ref{level}.                 
\end{proof}


\begin{thebibliography}{99}



 \bibitem{AB} W. Allegretto and A. Barabanova, 
 Positivity of solutions of elliptic equations with nonlocal terms,   
 {\it Proceedings of the Royal Society of Edinburgh A}, 
 {\bf 126}, n.3, (1996), 643--663.

  \bibitem{ACS} C.O. Alves, F.J.S.A. Correa and J.R. Santos J\'unior, 
 Remark on a class of integro-differential problems,   
 {\it J. Math. Anal. Appl.}, 
 {\bf 506}, (2022), 125723.

\bibitem{ASHR} 
W. Allegretto, B. Shen, P. Haswell and A.M. Robinson, 
Numerical modelling and optimization of micromachined thermal conductivity pressure sensor, 
{\it IEEE Transactions on Computer-Aided Design of Integrated Circuits and Systems},
{\bf 13}, n.10, (1994), 1247--1256.

\bibitem{BPS} M.A.L. Bahia, M.T.O. Pimenta and J.R. Santos J\'unior, Existence, sign and asymptotic behaviour for a class of integro-differential elliptic type problems, Calc. Var. PDE {\bf 63}, (2024), 122.


\bibitem{BC} V. Benci and G. Cerami, Multiple positive solutions of some elliptic problems via the Morse theory and the domain topology, Calc. Var. PDE {\bf 2}, (1994), 29-48.

  \bibitem{BZ} K.J. Brown and Y. Zhang, The Nehari manifold for a semilinear elliptic equation with a sign-changing weight function, {\it J. Differential Equations}, {\bf 193}(2), (2003), 481--499. DOI: \href{https://doi.org/10.1016/S0022-0396(03)00121-9}{10.1016/S0022-0396(03)00121-9}.

  \bibitem{Cat}
  E. A. Catchpole,
  A Cauchy problem for an ordinary integro-differential equation,
  {\it Proceedings of the Royal Society of Edinburgh A},
  {\bf 72}, (1972/73) 39--55.


  \bibitem{Cha}
N. Chafee,
The electric ballast resistor: homogeneous and nonhomogeneous equilibria,
  {\it Nonlinear Differential Equations},
 pp. 97--127,
 New York: Academic Press, 1981.
 
  \bibitem{CLL} F.J.S.A. Correa, N.A. Lima and R.N. de Lima, 
 Existence of solutions of integro-differential semilinear elliptic equations,   
 {\it Applicable Analysis}, 
 (2021), https://doi.org/10.1080/00036811.2021.2005786.

   \bibitem{dLd}
  G.C.G. dos Santos, N.A. Lima and R.N. de Lima,
  Existence and multiple of solutions for a class integro-differential equations with singular term via variational and Galerkin methods,
  {\it Nonlinear Analysis: Real World Applications},
  {\bf 69}, (2023), 103752.

\bibitem{Eke} I. Ekeland, On the variational principle. {\it J. Math. Anal. Appl.}, {\bf 47}, (1974), 324--353.

\bibitem{FP}
B. Fiedler and P. Pol\'a\u{c}ik, 
Complicated dynamics of scalar reaction diffusion equations with a nonlocal term, 
{\it Proceedings of the Royal Society of Edinburgh Section A: Mathematics}
{\bf 115}, n.1-2, (1990), 167--192.



\bibitem{FQS} G.M. Figueiredo, H. Ramos Quoirin and K. Silva, Ground states of elliptic problems over cones, Calc. Var. PDE {\bf 60}, (2021), 189.

\bibitem{FS} G.M. Figueiredo and J.R. Santos Júnior, Multiplicity and concentration behaviour of positive solutions for a Schrodinger-Kirchhoff type problem via penalization method, ESAIM: COCV {\bf 20}, (2014), 389–415.

\bibitem{PF}
P. Freitas, 
A nonlocal Sturm-Liouville eigenvalue problem, 
{\it Proceedings of the Royal Society of Edinburgh Section A: Mathematics}
{\bf 124}, (1994), 169--188. 

\bibitem{LZ} G. Li and H.-S. Zhou, Multiple solutions to $p$-laplacian problems with asymptotic nonlinearity as $u^{p-1}$ at infinity, {\it J. London Math. Soc.}, (2){\bf65}(2002)123-138.


\bibitem{LWW} J. Liu, Y. Wang and Z. Wang, Solutions for quasilinear Schrödinger equations via the Nehari method, Comm. Partial Differ. Equ. {\bf 29(5-6)}, (2004), 879–901.

\bibitem{NPS} J.R.S. Nascimento, M.T.O. Pimenta and J.R. Santos Júnior, A Nehari approach for asymptotically linear elliptic problems, {\it Bull. Braz. Math. Soc. (N.S.)}, 56, Art. 32 (2025). DOI: \href{https://doi.org/10.1007/s00574-025-00457-4}{10.1007/s00574-025-00457-4}.

\bibitem{OSlog} P. Omari and E. Sovrano, Positive solutions of indefinite logistic growth models with flux-saturated diffusion, {\it Nonlinear Anal.}, {\bf 201}, (2020), 111949. DOI: \href{https://doi.org/10.1016/j.na.2020.111949}{10.1016/j.na.2020.111949}.

\bibitem{OSsuper} P. Omari and E. Sovrano, Positive solutions of superlinear indefinite prescribed mean curvature problems, {\it Commun. Contemp. Math.}, {\bf 23}(3), (2021), 2050017. DOI: \href{https://doi.org/10.1142/S0219199720500170}{10.1142/S0219199720500170}.

\bibitem{PZ} K. Perera and Z. Zhang, Nontrivial solutions of Kirchhoff-type problems via the Yang
index, Journal of Differential Equations {\bf 221(1)}, (2006), 246–255.

\bibitem{QZ} W. Qi and X. Zhang, Multiplicity of solutions for a nonhomogeneous quasilinear elliptic equation with concave-convex nonlinearities, {\it Adv. Nonlinear Anal.}, {\bf 13}(1), (2024), 20240047. DOI: \href{https://doi.org/10.1515/anona-2024-0047}{10.1515/anona-2024-0047}.

\bibitem{RQS} H. Ramos Quoirin and K. Silva, Local minimizers for a class of functionals over the Nehari set, {\it J. Math. Anal. Appl.}, {\bf 519}(2), (2023), 126851. DOI: \href{https://doi.org/10.1016/j.jmaa.2022.126851}{10.1016/j.jmaa.2022.126851}.

\bibitem{SW} A. Szulkin and T. Weth, The method of Nehari manifold, {\it Handbook of Nonconvex Analysis and Applications}. D.Y. Gao and D. Motreanu eds., International Press, Boston, (2010), 597-632.

\end{thebibliography}
\end{document}